\documentclass[twoside,final]{amsart}

\usepackage{amsfonts,amsmath,amscd,amsthm,amssymb}
\usepackage{mathtools}
\usepackage[mathscr]{euscript} 
\usepackage{graphics}
\usepackage{graphicx}
\usepackage{wrapfig}
\usepackage{lscape}
\usepackage{rotating}
\usepackage{epsfig}
\usepackage{cite}
\usepackage{color}
\usepackage[notcite,notref]{showkeys}
\usepackage[vlined,ruled]{algorithm2e}
\usepackage{multirow}
\usepackage{subfig}

\definecolor{xxx}{rgb}{0.8,0,0}

\definecolor{red}{rgb}{0.8,0,0}

\newtheorem{lemma}{Lemma}
\newtheorem{theorem}{Theorem}

\newtheorem{remark}{Remark}

\renewcommand{\b}[1]{{\boldsymbol{#1}}}

\def\restrict#1{\raise-0.2ex\hbox{\ensuremath|}_{#1}}

\newcommand{\dK}{{\partial K}}

\newcommand{\bR}{\mathbb R}

\newcommand{\bld}[1]{\boldsymbol{#1}}
\newcommand{\curls}{{{\nabla\times}}}
\newcommand{\divs}{{\nabla\cdot}}
\newcommand{\grads}{{\nabla}}

\newcommand{\pol}{\mathbb{P}}

\newcommand{\vertiii}[1]{{\left\vert\kern-0.25ex\left\vert\kern-0.25ex\left\vert #1 
    \right\vert\kern-0.25ex\right\vert\kern-0.25ex\right\vert}}

\newcommand{\Oh}{\mathcal{T}_h}
\newcommand{\Eh}{\mathcal{E}_h}

\newcommand{\EK}{\mathcal{E}(K)}
\newcommand{\EKP}{\mathcal{E}(K')}
\newcommand{\EKI}{\mathcal{E}^o(K)}
\newcommand{\EKB}{\mathcal{E}^\partial(K)}
\newcommand{\bintK}[2]{\langle #1\,,\,#2 \rangle_{\partial{K}}}
\newcommand{\bint}[3]{\langle #1\,,\,#2 \rangle_{#3}}
\newcommand{\bintF}[2]{\langle #1\,,\,#2 \rangle_{F}}

\newcommand{\Vhk}{V_{h,k}}
\newcommand{\Shk}{\Sigma_{h,k}}
\newcommand{\Mhk}{M_{h,k}}
\newcommand{\CV}{\mathbb{V}}
\newcommand{\CX}{\mathbb{X}}

\newcommand{\jmp}[1]{\,[\![#1]\!]}
\newcommand{\avg}[1]{\,\{#1\}}

\begin{document}
\title[posteriori HDG]{Fully computable a posteriori error bounds for hybridizable discontinuous Galerkin finite element 
approximations}
\author{Mark Ainsworth}
\address{Division of Applied Mathematics, Brown University, 182 George St,
Providence RI 02912, USA.}
\email{Mark\_Ainsworth@brown.edu}
\thanks{First author gratefully acknowledges the partial support of this work
under AFOSR contract FA9550-12-1-0399.}

\author{Guosheng Fu}
\address{Division of Applied Mathematics, Brown University, 182 George St,
Providence RI 02912, USA.}
\email{Guosheng\_Fu@brown.edu}


\begin{abstract}
We derive a posteriori  error estimates for the hybridizable discontinuous Galerkin (HDG) methods, including both 
the primal and mixed formulations, 
for the approximation of a linear second-order elliptic problem on conforming simplicial meshes in two and three dimensions.

We obtain fully computable, constant free, a posteriori error bounds on 
the broken energy seminorm and the HDG energy (semi)norm of the error.
The estimators are also shown to provide local lower bounds for the HDG energy (semi)norm of the error up to a constant and a higher-order data oscillation term. For the primal HDG methods and mixed HDG methods with an appropriate choice of stabilization parameter, the 
 estimators are also shown to provide a lower bound for the broken energy seminorm of the error up to a constant and a higher-order data oscillation term.
 Numerical examples are given illustrating the theoretical results.
\end{abstract}
\maketitle

\section{Introduction}
Recent years have seen the developments of fully computable, guaranteed error bounds for 
the conforming \cite{LadevezeLeguillon83,Kelly84,AinsworthOden93,AinsworthOden00, DestuynderMetivet99,LuceWohlmuth04,NicaiseWitowskiWohlmuth08,BraessPillweinSchoberl09, Vohralik08, CarstensenMerdon10,CaiZhang12}, 
nonconforming \cite{DestuynderMetivet98,Ainsworth05,Kim07a, Kim07b, ErnVohralik13, CarstensenMerdon13, AinsworthRankin08}, discontinuous Galerkin
\cite{Ainsworth07a,Kim07a, Kim07b,CochezDhondtNicaise08, AinsworthRankin10a, AinsworthRankin11b}, and mixed finite element methods 
\cite{Ainsworth07b,Kim07a, AinsworthMa12, CockburnZhang14}; see also {\it unified frameworks} in \cite{Ainsworth10a,CarstensenEigelHoppeLobhard12,ErnVohralik15}. 

In comparison, 
there are relatively few works on a posteriori error estimates for the hybridizable discontinuous Galerkin (HDG) methods \cite{CockburnGopalakrishnanLazarov09}. 
The a posteriori estimates for HDG methods  that are currently available in the literature
\cite{CarstensenHoppeSharmaWarburton11,CockburnZhang12,CockburnZhang13,EggerWaluga13,ChenLiQiu16,GaticaSequeira16} are all of residual type, in which reliability is shown up to a generic (unknown) constant.
This means that, while the associated estimation may be suitable as local refinement indicators, they cannot provide a quantitative stopping criterion. Moreover, if only a single fixed mesh is used (as is often the case in practice) then the value of an a posteriori bound containing unknown constants is somewhat questionable.
Here we present, for the first time, fully computable a posteriori error bounds for HDG methods, for both the primal and mixed formulations, in the setting of a linear second-order elliptic problem on conforming simplicial meshes in two and three space dimensions. 
The key ingredient of our analysis is the local conservation property of the HDG methods, with which cheap element-wise {\it equilibrated fluxes} can be constructed.

The remainder of this article is organized as follows. Section \ref{sec:pre} 
presents the model problem and prepares the notation used throughout the article. In Section \ref{sec:primal}, we introduce the primal HDG schemes and the corresponding computation error bounds.
While in Section \ref{sec:mixed}, 
we introduce the mixed HDG schemes and the corresponding computation error bounds.
Numerical results are then presented in Section \ref{sec:numerics}. 
The proofs of the main results in Section 3 and Section 4 are presented in Section \ref{sec:proofs}.

\section{Preliminaries}
\label{sec:pre}
\subsection{Model Problem}
Consider the following model problem:
\begin{align}
\label{pde}
 \left.
\begin{tabular}{r l}
 $\b\sigma- a\,\grads u=$&\; \!\!\!\!$0$ \\
 $-\divs \b\sigma=$&\; \!\!\!\!$f\in L^2(\Omega)$ \\
\end{tabular}
 \right\}\text{ in }\Omega
\end{align}
subject to $u=0$ on $\partial \Omega$, where $\Omega\in \mathbb{R}^d$, $d\in\{2,3\}$, is a simply connected polygonal/polyhedral domain. 
The datum $a\in L^\infty(\Omega)$ is assumed to be strictly positive and, for simplicity, is assumed piecewise constant on subdomains of $\Omega$.

\subsection{Notation and finite elements}
We consider a family of partitions
$\Oh = \{K\}$ of the domain $\Omega$ 
into the union of nonoverlapping, shape-regular, simplicial elements such that
the nonempty intersection of a distinct pair of elements is a single common node,
single common edge or single common face (in three dimensions). The family of partitions is assumed to be locally quasi-uniform
in the sense that the ratio of the diameters of any pair of neighboring elements is
uniformly bounded above and below over the whole family. 
Furthermore, it is assumed that the partitioning is compatible with the data so that $a$ is  constant on
each element.

The set of all facets (edges in two dimensions and faces in three dimensions) 
of the elements is denoted by $\Eh$, which we partition into subsets
$\Eh^\partial$ and $\Eh^o$ consisting of facets lying on the boundary $\partial \Omega$, 
and the remaining interior facets, respectively. 
Likewise, the corresponding quantities relative to an individual element $K$ are denoted by 
$\EK, \EKB,$ and
$\EKI$, respectively.
For each facet $F \in \Eh$, the set $\widetilde F$ consists of those elements for which $F$ is a facet,
\begin{align}
\label{facet-union}
\widetilde F = \{K'\in \Oh:\; F\in \EKP \}, 
\end{align}
while, for each element $K\in \Oh$, the set $\widetilde K$ consists of those elements having a facet in common with $K$,
\begin{align}
\label{element-union}
\widetilde K = \{K'\in \Oh:\; \EK\cap \EKP\text{ is nonempty} \}.
\end{align}

Let $h_D$ denote the diameter of a domain $D$,
 $|K|$ denote the measure of an element $K$,
 $|F|$ denote the measure of a facet $F$, and 
 $|\partial K|= \sum_{F\in\EK} |F|$ denote the measure of the boundary of an element $K$.

Let 
 $\Shk, \Vhk$, and $\Mhk$ denote the finite dimensional spaces
\begin{subequations}
\label{fem-space}
\begin{alignat}{3}
 \label{fem-s}
 \Sigma_{h,k} := &\;\{\b\tau \in L^2(\Oh)^d: &&\;\;\b\tau\restrict K \in\pol_k(K)^d\quad\forall 
 K\in \Oh\},
 \\
 \label{fem-v}
 V_{h,k} :=&\; \{v\in L^2(\Oh): &&\;\;v\restrict{K}\in \pol_k(K)\;\;\forall K\in \Oh\},\\
 \label{fem-m}
 M_{h,k} := &\;\{\widehat v\in L^2(\Eh): &&\;\;\widehat v\restrict{F}\in \pol_k(F)\;\;\forall F\in \Eh\},
\end{alignat}
where $\pol_k(D)$ denotes the set of 
polynomials of degree at most $k\ge 0$ on the domain $D$.
The space of homogeneous polynomials of degree $k$ on a domain $D$ is denoted as $\widetilde{\pol}_{k}(D)$. 
We shall also need the subspace $ M_{h,k}^0\subset \Mhk$ given by 
\begin{align}
 \label{primal-space-m0}
 M_{h,k}^0 := \{\widehat v\in \Mhk:\;\;\widehat v\restrict F = 0\;\;\forall F\in \Eh^\partial
 \}.
\end{align}
\end{subequations}

To simplify notation,  we introduce the compound finite-dimensional space
\begin{align}
 \label{primal-compound}
 \CV_{h,k, \delta} :=  &\; 
  V_{h,k}\times  M_{h,k-\delta}^0, \quad k\ge 1, \delta \in\{0,1\},
\end{align}
which is used for the primal HDG scheme,
while the compound finite-dimensional space
\begin{align}
 \label{mixed-compound}
 \CX_{h,k} :=  &\; 
 \Sigma_{h,k}\times V_{h,k}\times  M_{h,k}^0, \quad k\ge 0,
\end{align}
 is used for the mixed HDG scheme.

 {The stabilization parameters for the HDG schemes will be taken from the  space 
$M_{h,0}^{\mathrm{dc}}$, where, for any nonnegative integer $m$,
\begin{align}
 \label{fes-stab}
 M_{h,m}^{\mathrm{dc}} : = \Pi_{K\in\Oh}\pol_m(\dK),
\end{align}
with $\pol_m({\dK}) := \{\mu\in L^2(\dK):\;\;
\mu\restrict F \in\pol_m(F)\;\;\forall F\in \EK\}.$}

We use the standard notation for jumps and averages \cite{ArnoldBrezziCockburnMarini02} of functions in $M_{h,m}^{\mathrm{dc}}$ and 
$[M_{h,m}^{\mathrm{dc}}]^d$ on the mesh skeleton $\Eh$:
Let $F$ be a facet shared by elements $K^+$ and $K^-$ with  unit normal vectors $\b n^+$ and 
$\b n^-$ on $F$ pointing exterior to $K^+$ and $K^-$ respectively, then 
for $q \in M_{h,m}^{\mathrm{dc}}$
\begin{subequations}
\label{avg-jmp}
\begin{align}
\label{avg-jmp-1}
 \avg{q} = \frac{1}{2}(q\restrict{K^+}+q\restrict{K^-}), 
\quad \jmp{q} = q\restrict{K^+}\b n^++q\restrict{K^-}\b n^-
\quad \text{ on } F\in \Eh^o.
\end{align}
Similarly, for $\b\phi\in [M_{h,m}^{\mathrm{dc}}]^d$ we set
\begin{align}
\label{avg-jmp-2}
 \avg{\b\phi} = \frac{1}{2}(\b\phi\restrict{K^+}+\b\phi\restrict{K^-}), 
\quad \jmp{\b\phi} = \b\phi\restrict{K^+}\cdot \b n^++\b\phi\restrict{K^-}\cdot\b n^-
\quad \text{ on } F\in \Eh^o.
\end{align}
On a boundary facet $F\in \Eh^\partial$, for each $q\in M_{h,m}^{\mathrm{dc}}$ we set 
\begin{align}
\label{avg-jmp-3}
 \avg{q} = q, 
\quad \jmp{q} = q\b n\quad \text{ on } F\in \Eh^\partial.
\end{align}
\end{subequations}

{We use the notation $(\cdot,\cdot)_{\omega}$ to denote the integral inner product over a region ${\omega}\subset \bR^d$, 
and $\|\cdot\|_{\omega}$ to denote the corresponding $L^2$-norm. 
We omit the subscript in the case when $\omega$ is the physical domain $\Omega$.
Finally, for each element $K\in\Oh$, we use the notation
$\bintK{\cdot}{\cdot}$ to denote the integral inner product over the element boundary $\dK$.
}

\section{The primal HDG methods and computable error bounds}
\label{sec:primal}
\subsection{Primal HDG formulation}

\newcommand{\Bpr}{\mathcal{B}_h^{pr}}
\newcommand{\Bmx}{\mathcal{B}_h^{mx}}
\newcommand{\Lpr}{\mathcal{L}_h^{pr}}
\newcommand{\Lmx}{\mathcal{L}_h^{mx}}
Let $\alpha_h\in M_{h,0}^{\mathrm{dc}}$ 
be a positive stabilization parameter 
to be specified later,
and define the bilinear form 
$\Bpr: \CV_{h,k,\delta}\times \CV_{h,k,\delta}
\rightarrow \mathbb{R}$ by
\begin{align}
 \label{primal-bilinear}
 \Bpr\big((u,\widehat u),(v,\widehat v)\big) =&\; \sum_{K\in\Oh}\Big\{
(a\,\grads u, \grads v)_K - \bintK{a\,\grads u\cdot \b n}{v-\widehat v\,}\\ 
&\!\!\!\!\!\!
\!\!\!\!\!\!
\!\!\!\!\!\!
\!\!\!\!\!\!
\!\!\!\!\!\!- \bintK{a\,\grads v\cdot \b n}{u-\widehat u }
+ \bintK{\alpha_h(P_M u-\widehat u)}{P_Mv-\widehat v\,} 
 \Big\}\nonumber,
\end{align}
where
$P_M$ denotes the $L^2$ projection onto the space $M_{h,k-\delta}^{\mathrm{dc}}$. 
We also define  the linear form $\Lpr: \CV_{h,k,\delta}\rightarrow \mathbb{R}$ by  
\begin{align}
 \label{primal-linear}
 \Lpr\big((v,\widehat v)\big) =&\; \sum_{K\in\Oh}(f, v)_K. 
\end{align}

Let $u$ be the solution of \eqref{pde}.
An approximation of $(u, u\restrict{\Eh})$ is obtained by seeking 
$(u_h,\widehat u_h)\in  \CV_{h,k,\delta}$ such
that
\begin{align}
 \label{primal-hdg}
 \Bpr\big((u_h,\widehat u_h), (v_h,\widehat v_h)\big)
 =
 \Lpr\big((v_h,\widehat v_h)\big)
 \quad \forall 
 (v_h,\widehat v_h)\in  \CV_{h,k,\delta}.
\end{align}
This scheme is known as the 
{\it hybridized, symmetric, interior penalty discontinuous Galerkin method} \cite{Lehrenfeld:10}.

\subsection{The choice of the stabilization 
parameter $\alpha_h$.
}
It is well-known \cite{Lehrenfeld:10} that \eqref{primal-hdg} is well-posed provided the stabilization parameter is ``sufficiently large''. 
The following result quantifies exactly how large $\alpha_h$ must be; similar results for  interior penalty discontinuous Galerkin methods 
can be found in \cite{Shahbazi05, Ainsworth07a, EpshteynRiviere07,AinsworthRankin10a,AinsworthRankin11b,AinsworthRankin12a}.
\begin{lemma}
\label{lemma:primal-stab}
Suppose the  stabilization parameter 
$\alpha_h$ is given by  
\begin{align}
 \label{stab-pr}
\alpha_h\restrict{F} = \frac{a\restrict{K}\,\gamma}{|F|}\sum_{F'\in\EK} \frac{|F'|^2}{|K|}\quad\quad
\text{ for all } F\in\EK, \text{ for all } K\in\Oh,
\end{align}
where $\gamma$ is a constant satisfying
\begin{align}
 \label{stab-cst}
\gamma > \frac{k(k+d-1)}{d}.
\end{align}
Then \eqref{primal-hdg} has a unique solution $(u_h,\widehat u_h)\in \CV_{h,k, \delta}$ 
for $k\ge 1$ and 
$\delta \in\{0,1\}$.
\end{lemma}
{The proof of this and other results in this section is postponed to Section \ref{sec:proofs}.}
%
\begin{remark}
Observe that the stabilization parameter \eqref{stab-pr} on a facet $F$ is 
proportional to $h_F^{-1}$, which gives 
an optimal order a priori convergence rate $\mathcal{O}(h^k)$ in the energy norm \cite{Lehrenfeld:10,Oikawa:15}.
\end{remark}


\subsection{The broken energy seminorm and the HDG energy norm}
For a given function $(v,\widehat v)\in \CV_{h,k,\delta} + 
(H^1(\Omega)\times L^2(\Eh))$, 
let the broken energy seminorm $\vertiii{\cdot}$
be denoted by
\begin{align}
 \label{broken-h1}
 \vertiii{(v,\widehat v)}_{pr} = \left(\sum_{K\in\Oh}(a\grads v, \grads v)_K\right)^{1/2}.
\end{align}
Our objective is to derive a fully computable estimator for the error in the HDG finite-element approximation 
$(e_u, \widehat e_u) = (u - u_h, u\restrict{\Eh}-\widehat u_h)$, 
where $u$ is the solution to \eqref{pde} and $(u_h,\widehat u_h)$ 
is the solution to \eqref{primal-hdg}.
Let the HDG energy norm 
$\vertiii{\cdot}_{HDG,pr}$
be
denoted by
\begin{align}
 \label{broken-hdg}
 \vertiii{(v,\widehat v)}_{HDG,pr} = \left(
  \vertiii{(v,\widehat v)}_{pr}^2 + \sum_{K\in\Oh} 
  \bintK{\alpha_h P_M(v-\widehat v)}{P_M(v-\widehat v)}
 \right)^{1/2}.
\end{align}
Observe that, since $P_M(e_u - \widehat e_u) = -P_M(u_h - \widehat u_h)$, the quantity 
\begin{align}
\label{jmp-primal}
 \sum_{K\in\Oh} 
  \bintK{\alpha_h P_M(e_u-\widehat e_u)}{P_M(e_u-\widehat e_u)}
\end{align}
is directly computable in terms of the HDG approximation $(u_h,\widehat u_h)$.
Hence, given a constant free estimator for $\vertiii{(e_u,\widehat e_u)}_{pr}$, 
we automatically have a constant free estimator for the HDG energy norm of the
error as well. 
The next result shows that, by analogy with the standard interior penalty methods \cite{Ainsworth07a,AinsworthRankin11b}, 
these norms are equivalent in the following sense.
\begin{lemma}
 \label{lemma:primal-norm-eq}
Let the stabilization parameter be given by \eqref{stab-pr} 
with the global constant $\gamma$ satisfying \eqref{stab-cst}, 
then the HDG energy norm and the broken energy seminorm 
of the error $(e_u, \widehat e_u) = (u - u_h, u\restrict{\Eh}-\widehat u_h)$
are equivalent. That is to say,
\begin{align}
 \vertiii{(e_u,\widehat e_u)}_{pr}\le \vertiii{(e_u, \widehat e_u)}_{HDG,pr},
\end{align}
and there exists a positive constant $c$, 
depending only on the shape-regularity of the mesh, the polynomial degree $k$, and the local permeability ratio 
between neighboring elements,  such that
\begin{align}
\label{norm-eq2}
c \vertiii{(e_u, \widehat e_u)}_{HDG,pr}^2\le 
\vertiii{(e_u, \widehat e_u)}_{pr}^2
+ \sum_{K\in\Oh} osc_{k-1}^2(f,K).
\end{align}
where the data oscillation on an element $K \in\Oh$ is defined
to be
\begin{align}
\label{osc}
 osc_m(f,K) = a^{-1/2}h_K
 \inf_{p \in \pol_m}\| f -p\|_K. 
\end{align}
\end{lemma}
{The proof of this result is postponed to Section \ref{sec:proofs}.}

\subsection{Local conservation}
The numerical flux is defined by 
\newcommand{\spr}{  \widehat{\b\sigma}_{h,pr}}
\newcommand{\smx}{  \widehat{\b\sigma}_{h,mx}}
\newcommand{\sprn}{  \widehat{\b\sigma}_{h,pr}\cdot{\bld n}}
\newcommand{\smxn}{  \widehat{\b\sigma}_{h,mx}\cdot{\bld n}}
\[
  \spr: = 
  a\grads u_h-\alpha_h(P_Mu_h-\widehat u_h) \b n\in [M_{h,k-\delta}^{\mathrm{dc}}]^d
 \]
and  satisfies the local conservation property
 \begin{align}
  \label{local-c-p}
(f,1)_K+\bintK{\sprn}{1} = 0
 \quad \text{ for each element $K\in\Oh$},
 \end{align}
%
{along with
 \begin{align}
 \label{jmp-cont-p}
\jmp{\spr}\restrict F = 0 \quad 
\text{ on each interior facet }  F\in \Eh^o.
 \end{align}
}
 
 These results are a straightforward consequence of the definition and \eqref{primal-hdg}.

\subsection{The computable error bounds}
We obtain 
computable error bounds for 
the discrete energy norm of the error by bounding the conforming and non-conforming errors separately \cite{Ainsworth07a,AinsworthRankin11b, Ainsworth10a}.
To this end, two types of post-processing scheme will be needed.
\subsubsection{Local (equilibrated) flux post-processing}
We define a 
local flux post-processing \cite{Ainsworth07a, AinsworthRankin10a} as follows:
Let $\b\sigma_h^*\in \Sigma_{h,k}$ be such that,
on each element $K$, there holds
\begin{subequations}
\label{primal-flux-p}
\begin{alignat}{3}
\label{primal-flux-p-1}
(\divs \b\sigma_h^*, v)_K = &\; 
-(f , v)_K
&&\quad \forall\, v \in \pol_{k-1}(K) \text{ and } (v,1)_K=0, \\
\label{primal-flux-p-2}
\bintK{\b\sigma_h^*\cdot\b n}{\widehat v}
 = &\; 
 \bintF{\sprn}{\widehat v}
 &&\quad \forall\, \widehat v \in \pol_{k}(F),\quad \forall F\in \EK,\\
\label{primal-flux-p-3}
(\b\sigma_h^*, \b\tau)_K = &\; 
(a\grads u_h, \b\tau)_K &&\quad \forall\, \b\tau \in 
\Sigma_{k, sbb}(K),
\end{alignat}
\end{subequations}
where 
\begin{align}
\label{div-free}
\Sigma_{k,sbb}(K) :=
\{\b\tau \in \pol_k(K)^d:\; 
\divs \b\tau = 0, \;\; 
\b\tau\cdot\b n\restrict F  = 0 \text{ for all }F\in\EK
\} 
\end{align}
 denotes the divergence-free ``bubble'' space.
 The unique solvability of \eqref{primal-flux-p} can be established using arguments similar to those used to study the closely related 
Brezzi-Douglas-Marini (BDM) projection \cite{BoffiBrezziFortin13}.
By the local conservation properties \eqref{local-c-p} and \eqref{jmp-cont-p}, we conclude that 
$  \b\sigma_h^* \in \Sigma_{h,k}\cap H(\mathrm{div};\Omega)$ satisfies
\begin{align}
\label{eq-p}
\divs \b\sigma_h^* = -\Pi_{k-1}f \text{ on $\Oh$}, 
 \text{ and }
 \b\sigma_h^*\cdot\b n =  \sprn \text{ on $\partial \Oh$}.
\end{align}
The quantity $\b\sigma_h^*$ is usually referred to as an {\it equilibrated flux} \cite{AinsworthOden00}.

\subsubsection{Local  potential post-processing}
We obtain a globally continuous potential 
by a simple averaging \cite{Ainsworth07a,AinsworthRankin11b} of the discontinuous potential $u_h$ as follows:
Let $\mathcal{N}_{K,k}$ index a set of points $\{\b x_m\}_{m\in\mathcal{N}_{K,k}}$
on $\overline K$ associated with
a Lagrange basis for the conforming finite-element space 
of order $k$ on $\Oh$ 
and let $\mathcal{N}_{K,k}^o$
denote the restriction of the set $\mathcal{N}_{K,k}$ to the points that do not lie on the boundary of element $K$, with $\mathcal{N}_{\gamma,k}$ being its complementary set.
Let $N_{\gamma,k}^\partial$ denote the
restriction of the set $\mathcal{N}_{\gamma,k}$ to the points that lie on the closure of the  boundary $\partial \Omega$. 
For $m\in \mathcal{N}_{\gamma,k}$, let 
$\Omega_m$ denote the set of elements in $\Oh$ whose closure contains the point $\b x_m$.

The post-processed  potential
$u_h^*\in V_{h,k}\cap 
H^1(\Omega)
$
is obtained through a  simple averaging of the degrees of freedom for $u_h$:
for all elements $K\in\Oh$, $u_h^*\restrict K = S_k(u_h)\restrict K \in\pol_k(K)$, where the nodal values are given by
\begin{align}
\label{primal-potential-p}
S_k(u_h)(\b x_m) =\left\{ \begin{tabular}{l l}
                 $0$ & if $m \in \mathcal{N}_{\gamma,k}^\partial$,
                 \vspace{0.12cm}\\
                 $u_h(\b x_m)$ & if $m \in \mathcal{N}_{K,k}^o$,
                 \vspace{0.12cm}\\
                 $\frac{1}{\# \Omega_m}\sum_{K'\in \Omega_m}
                 u_{h}\restrict {K'}(\b x_m)
                 $ & if $m \in \mathcal{N}_{\gamma,k} \backslash\mathcal{N}_{\gamma,k}^\partial$, 
                \end{tabular}
                \right.
\end{align}
and $\# \Omega_m$ denotes the number of elements of $\Oh$ contained in the patch $\Omega_m$.

\subsubsection{Computable error bounds}
The foregoing developments show that each of the quantities
\begin{subequations}
\label{indicator-p}
\begin{align}
\label{indicator-p-1}
\eta_{CF, K} = &\; \|a^{-1/2}(\b\sigma_h^*-a\grads u_h)\|_K
+
\frac{h_K}{\pi}a\restrict K^{-1/2}\, \|f-\Pi_{k-1}f\|_K
,\\
\label{indicator-p-2}
\eta_{NC, K} = &\; \|a^{1/2}\grads (u_h-u_h^*)\|_K
\end{align}
\end{subequations}
can be computed directly from the primal HDG approximation using purely local computations. 
The next result shows that together
these quantities provide a computable, constant-free, upper bound on  
the broken energy seminorm of the
error:

\begin{theorem}
\label{thm:primal}
 Let 
$ \eta_{CF, K}$ and $\eta_{NC, K}$ be defined as in \eqref{indicator-p}, 
 and let the stabilization parameter $\alpha_h$  be of the form \eqref{stab-pr} where $\gamma$ 
 satisfies \eqref{stab-cst}. 
 Then 
\begin{align}
 \label{broken-energy-estimate}
 \vertiii{(e_u,\widehat e_u)}_{pr}^2
 \le \eta^2 = \sum_{K\in\Oh} (\eta_{CF,K}^2+\eta_{NC,K}^2).
\end{align}
Moreover, there exists a positive constant $c$, 
depending only on the shape-regularity of the mesh, the polynomial degree $k$, and the local permeability ratio 
between neighboring elements,  such that
\begin{align}
 \label{broken-energy-estimate2}
c\sum_{K\in\Oh} (\eta_{CF,K}^2+\eta_{NC,K}^2)\le \vertiii{(e_u,\widehat e_u)}_{pr}^2
+\sum_{K\in\Oh} osc_{k-1}^2(f,K).
\end{align}
Furthermore, 
\begin{align*}
 \vertiii{(e_u,\widehat e_u)}_{HDG,pr}^2
 \le \eta_{HDG}^2 = \eta^2 +
 \sum_{K\in\Oh}\bintK{\alpha_h(P_M u_h-\widehat u_h)}{P_M u_h-\widehat u_h},
\end{align*}
and
\begin{align*}
c \,\eta_{HDG}^2\le 
 \vertiii{(e_u,\widehat e_u)}_{HDG,pr}^2
+\sum_{K\in\Oh} osc_{k-1}^2(f,K).
\end{align*}
where $c$ is as above. 
\end{theorem}
The proof of this result 
is similar to \cite[Theorem 2]{Ainsworth07a} and an outline of the main steps is given in Section \ref{sec:proofs} below.
Numerical examples illustrating the bounds in practice are given in Section \ref{sec:numerics}.
%

\section{The mixed HDG methods and computable error bounds}
\label{sec:mixed}
\subsection{The mixed HDG formulation}
Whereas the primal HDG method gives an approximation for $(u, u\restrict{\Eh})$, the mixed HDG method seeks, in addition, to approximate the flux $(a\grads u, u, u\restrict{\Eh})$.

Let $\alpha_h\in M_{h,0}^{d.c.}$ be 
a positive stabilization parameter to be specified later,
and define the bilinear form 
$\Bmx: \CX_{h,k}\times \CX_{h,k}
\rightarrow \mathbb{R}$ by
\begin{align}
 \label{mixed-bilinear}
 \Bmx\big((\b\sigma, u,\widehat u),(\b\tau, v,\widehat v)\big) =&\; \sum_{K\in\Oh}\Big\{
(a^{-1}\,\b\sigma, \b\tau)_K 
+(u, \divs\b\tau)_K 
- \bintK{\widehat u}{\b\tau\cdot\b n}\\ 
&\!\!\!\!\!\!
\!\!\!\!\!\!
\!\!\!\!\!\!
\!\!\!\!\!\!
\!\!\!\!\!\!+
(\b\sigma, \grads v)_K 
- \bintK{\b\sigma\cdot\b n-\alpha_h(u-\widehat u)}{v-\widehat v\,} 
 \Big\}\nonumber,
\end{align}
along with the linear form $\Lmx: \CX_{h,k}\rightarrow \mathbb{R}$ by  
\begin{align}
 \label{mixed-linear}
 \Lmx\big((\b\tau, v,\widehat v)\big) =&\; \sum_{K\in\Oh}(f, v)_K. 
\end{align}

An approximation of the true solution $(a\grads u, u, u\restrict{\Eh})$ is obtained by seeking 
$(\b\sigma_h, u_h,\widehat u_h)\in  \CX_{h,k}$ such
that
\begin{align}
 \label{mixed-hdg}
 \Bmx\big((\b\sigma_h, u_h,\widehat u_h), (\b\tau_h, v_h,\widehat v_h)\big)
 =
 \Lmx\big((\b\tau_h, v_h,\widehat v_h)\big)
 \quad \forall 
 (\b\tau_h, v_h,\widehat v_h)\in  \CX_{h,k}.
\end{align}
This  scheme was originally termed the 
{\it local discontinuous Galerkin-hybridizable method} (LDG-H) \cite{CockburnGopalakrishnanLazarov09} but is referred to here as 
the mixed HDG approximation.

\subsection{The choice of the stabilization 
parameter $\alpha_h$.
}
The mixed HDG scheme enjoys greater stability properties than the primal HDG scheme as 
the stabilization parameter $\alpha_h$ need only 
be {(partially)} positive 
in order for the scheme to be well-posed 
as the following result
\cite[Proposition 3.2]{CockburnGopalakrishnanLazarov09} shows.
\begin{lemma}
\label{lemma:mixed-stab}
  If the nonnegative 
  stabilization parameter 
$\alpha_h$  is chosen such that
$\alpha_h >0$ on at least one facet $F\in\EK$ for every element $K$, then there exists a unique solution $(\b\sigma_h, 
u_h,\widehat u_h)\in \CX_{h,k}$ 
for $k\ge 0$.
\end{lemma}
\begin{remark}[Stabilization parameter]
The two most common choices of stabilization parameter used in practice are: 
\begin{itemize}
 \item  uniform stabilization
 \begin{subequations}
 \label{mixed-stab}
  \begin{align}
  \label{mixed-stab1}
    \alpha_h\restrict{F} = &\;a\restrict K 
    && \text{ for all } F\in\EK,\;\;
    \text{ for all } K\in\Oh,
    \end{align}
 \item single-facet stabilization
    \begin{align}
  \label{mixed-stab2}
    \alpha_h\restrict{F} = &\;\left\{ \begin{tabular}{l l}
   $a\restrict K$ & if $F= F^*_K$\\                  
   $0$ & if $F\not =  F^*_K$       
                             \end{tabular}
                             \right.
                             && \text{ for all } F\in\EK,\;\;
    \text{ for all } K\in\Oh,
    \end{align}
 \end{subequations}
 where $F^*_K$ is an arbitrarily chosen but fixed facet of $K$.
\end{itemize}
 
Each of the above choices of stabilization parameters results in an 
optimal a priori convergence rate $\mathcal{O}(h^{k+1})$
of the error in the energy norm \cite{CockburnGopalakrishnanSayas10}. 

\end{remark}


\subsubsection{The broken energy seminorm and the HDG energy seminorm}
For a given function $(\b\tau, v,\widehat v)\in \CX_{h,k,\delta} + 
(H(\mathrm{div};\Omega)\times H^1(\Omega)\times L^2(\Eh))$, 
let the broken energy seminorm $\vertiii{\cdot}_{mx}$
be denoted by
\begin{align}
 \label{broken-h1-m}
 \vertiii{(\b\tau, v,\widehat v)}_{mx} = 
 \left(\sum_{K\in\Oh}(a^{-1} \b\tau, \b\tau)_K\right)^{1/2},
\end{align}
and let the mixed HDG energy seminorm 
$\vertiii{\cdot}_{HDG,mx}$
be
denoted by
\begin{align}
 \label{broken-hdg-m}
 \vertiii{(\b\tau, v,\widehat v)}_{HDG,mx} = \left(\!\!
  \vertiii{(\b\tau, v,\widehat v)}_{mx}^2 + \sum_{K\in\Oh} 
  h_K\bintK{\alpha_h (v-\widehat v)}{(v-\widehat v)}
 \!\!\right)^{1/2}.
\end{align}

The error in the mixed HDG finite-element approximation is denoted by 
$(\b e_\sigma, e_u, \widehat e_u) = 
(\b\sigma-\b\sigma_h, u - u_h, u\restrict{\Eh}-\widehat u_h)$, 
where $(\b\sigma, u)$ is the solution to \eqref{pde} and $(\b\sigma_h, u_h,\widehat u_h)$ 
is the solution to \eqref{mixed-hdg}.
Similarly to the primal HDG case, we have $e_u - \widehat e_u = -(u_h - \widehat u_h)$ and hence the quantity 
\[
\sum_{K\in\Oh} 
h_K  \bintK{\alpha_h(e_u-\widehat e_u)}{(e_u-\widehat e_u)}
\]
can be evaluated directly given the mixed HDG approximation.
Consequently, given a constant free estimator for the broken energy
seminorm of the error, we automatically have a constant free estimator for the HDG energy seminorm of the error as well. 
The next result shows that, by analogy with the primal HDG case, the HDG energy seminorm of the error is equivalent to the broken energy 
seminorm when the single-facet stabilization \eqref{mixed-stab2} is used. 

{However, the equivalence fails to hold if the uniform stabilization \eqref{mixed-stab1} 
is employed.
For instance, in the case of lowest order ($k=0$) approximation,
the discrete energy norm (plus the data oscillation)  cannot control the jump term 
$
 \sum_{K\in\Oh} 
  h_K\bintK{\alpha_h(u_h-\widehat u_h)}{(u_h-\widehat u_h)},
$
as shown by the counterexample constructed in \cite[Section 2.4.1]{CockburnNochettoZhang16}.
The situation for general $k\in \mathbb{N}$ remains  an open problem at this time.}

\begin{lemma}
 \label{lemma:mixed-norm-eq}
Let the stabilization parameter $\alpha_h$ given by  \eqref{mixed-stab2}, 
then the HDG energy seminorm and the broken energy seminorm 
of the error $(\b e_\sigma, e_u, \widehat e_u) = (
\b\sigma  - \b\sigma_h, u - u_h, u\restrict{\Eh}-\widehat u_h)$
are equivalent in the sense that
\begin{align}
\label{eq1}
 \vertiii{(\b e_\sigma, e_u, \widehat e_u)}_{mx}\le \vertiii{(\b e_\sigma, e_u, \widehat e_u)}_{HDG,mx},
\end{align}
and there exists a positive constant $c$, 
depending only on the shape-regularity of the mesh and the polynomial degree $k$,
such that
\begin{align}
\label{eq2}
c \vertiii{(\b e_\sigma, e_u, \widehat e_u)}_{HDG,mx}^2\le 
\vertiii{(\b e_\sigma, e_u, \widehat e_u)}_{mx}^2
+ \sum_{K\in\Oh} osc_k^2(f,K).
\end{align}
\end{lemma}
\begin{proof}
 The result follows at once from Lemma \ref{lemma:mixed-jmp} in Section \ref{sec:proofs} below.
\end{proof}


\subsection{Local conservation}
Similarly to the primal HDG scheme \eqref{primal-hdg},  
the mixed HDG scheme \eqref{mixed-hdg} is  locally conservative but this time in the sense that 
 the numerical flux 
 \[
  \smx: = 
 \b \sigma_h-\alpha_h(u_h-\widehat u_h)\bld n\in [M_{h,k}^\mathrm{dc}]^d
 \]
satisfies
 \begin{align}
  \label{local-c-m}
(f,1)_K+\bintK{\smxn}{1} = 0
 \quad \text{ for each element $K\in\Oh$},
 \end{align}
along with
 \begin{align}
 \label{jmp-cont-m}
\jmp{\smx}\restrict F = 0 \quad 
\text{ on each interior facet } F\in \Eh^o.
 \end{align}
%

\subsection{The computable error bounds}
We are now in a position to 
present computable error bounds for 
the discrete energy error.
While the basic approach is motivated by the technique used in \cite{Ainsworth07b, AinsworthMa12,CockburnZhang14} for the mixed methods, 
the post-processing technique needed for the mixed HDG case is quite different.
\subsubsection{Local (equilibrated) flux post-processing}
Let $\b\sigma_h^*\in \Sigma_{h,k+1}$ satisfy the following conditions,
on each element $K$:
\begin{subequations}
\label{mixed-flux-p}
\begin{alignat}{3}
\label{mixed-flux-p-1}
(\divs \b\sigma_h^*, v)_K = &\; 
-(f , v)_K
&&\quad \forall\, v \in \pol_{k}(K) \text{ and } (v,1)_K=0, \\
\label{mixed-flux-p-2}
\bintK{\b\sigma_h^*\cdot\b n}{\widehat v}
 = &\; 
 \bintF{\smxn}{\widehat v}
 &&\quad \forall\, \widehat v \in \pol_{k+1}(F),\quad \forall F\in \EK,\\
\label{mixed-flux-p-3}
(\b\sigma_h^*, \b\tau)_K = &\; 
(\b\sigma_h, \b\tau)_K &&\quad \forall\, \b\tau \in 
\Sigma_{k+1, sbb},
\end{alignat}
where the divergence-free ``bubble'' space 
$\Sigma_{k+1,sbb}$ is defined as in \eqref{div-free}.
\end{subequations}
Thanks to the local conservation properties \eqref{local-c-p} and \eqref{jmp-cont-p}, we conclude that 
$\b\sigma_h^* \in \Sigma_{h,k+1}\cap H(\mathrm{div};\Omega)$ satisfies
\begin{align}
\label{eq-m}
\divs \b\sigma_h^* = -\Pi_{k}f \text{ on $\Oh$}, 
 \text{ and }
 \b\sigma_h^*\cdot\b n =  \smxn \text{ on $\partial \Oh$}.
\end{align}

\subsubsection{Local potential post-processing}
A global continuous potential is constructed by  
averaging a higher order discontinuous approximation to the potential.
However, the averaging scheme is more involved than the one used in the primal case:
Firstly, we find $u_h^{*,\mathrm{dc}}\in V_{h,k+1}$ so that,
on each element $K$, there holds
\begin{subequations}
\label{mixed-potential-p}
\begin{alignat}{3}
\label{mixed-potential-p-1}
(a\,\grads u_h^{*,\mathrm{dc}}, \grads v)_K = &\; 
(\b\sigma_h ,\grads v)_K
&&\quad \forall\, v \in \pol_{k+1}(K),\\
\label{mixed-potential-p-2}
(u_h^{*,\mathrm{dc}}, 1)_K = &\; 
(u_h, 1)_K,
\end{alignat}
\end{subequations}

The continuous potential post-processing 
$u_h^* =  S_{k+1}(u_h^{*,\mathrm{dc}})\in V_{h,k+1}\cap 
H^1(\Omega)
$
is then a simple averaging of the degrees of freedom for $u_h^{*,\mathrm{dc}}$,
where $S_{k+1}(\cdot)$ is defined as in \eqref{primal-potential-p}.

\subsubsection{Computable error bounds}
Each of the quantities
\begin{subequations}
\label{indicator-m}
\begin{align}
\label{indicator-m-1}
\eta_{CF, K} = &\; 
\|a^{-1/2}(\b\sigma_h^*-\b\sigma_h)\|_K+
\frac{h_K}{\pi}a\restrict K^{-1/2}\, \|f-\Pi_{k}f\|_K,
\\
\label{indicator-m-2}
\eta_{NC, K} = &\; \|a^{-1/2}(\b\sigma_h-a\grads u_h^*)\|_K,
\end{align}
\end{subequations}
can be computed directly from the mixed HDG approximation using only local computations. 
These quantities provide computable, constant-free, upper bounds on the 
the broken energy seminorm of the
error $(\b e_\sigma, e_u,\widehat e_u)$:

\begin{theorem}
\label{thm:mixed}
 Let 
$ \eta_{CF, K}$ and $\eta_{NC, K}$ be defined as in \eqref{indicator-p},  
 with the stabilization parameter $\alpha_h$  chosen to be  either \eqref{mixed-stab1} or 
 \eqref{mixed-stab2}. Then 
\begin{align}
 \label{broken-energy-estimate-m}
 \vertiii{(\b e_\sigma, e_u,\widehat e_u)}_{mx}^2
 \le \eta^2 = \sum_{K\in\Oh} (\eta_{CF,K}^2+\eta_{NC,K}^2).
\end{align}
Moreover, there exists a positive constant $c$,
depending only on the shape-regularity of the mesh, the polynomial degree $k$, and the local permeability ratio 
between neighboring elements,  such that
\begin{align}
 \label{broken-energy-estimate2-m}
c\sum_{K\in\Oh} (\eta_{CF,K}^2+\eta_{NC,K}^2)\le&\; \vertiii{(\b e_\sigma, e_u,\widehat e_u)}_{mx}^2\nonumber\\
&\hspace{-2.5cm}
+\sum_{K\in\Oh} \left(
h_K\bint{\alpha_h(u_h-\widehat u_h)}{u_h-\widehat u_h}{\dK\backslash F_K^*}+
osc_k^2(f,K)\right),
\end{align}
where $F_K^*$ is an arbitrary but fixed facet of $K$.
Furthermore,
\begin{align*}
 \vertiii{(\b e_\sigma, e_u,\widehat e_u)}_{HDG,mx}^2
 \le \eta_{HDG}^2 = \eta^2 +
 \sum_{K\in\Oh}h_K\bintK{\alpha_h( u_h-\widehat u_h)}{ u_h-\widehat u_h}
\end{align*}
with 
\begin{align*}
c \,\eta_{HDG}^2\le 
 \vertiii{(e_u,\widehat e_u)}_{HDG,mx}^2
+\sum_{K\in\Oh} osc_k^2(f,K).
\end{align*}
\end{theorem}
The proof of Theorem \ref{thm:mixed} is presented in Section \ref{sec:proofs}.

\begin{remark}[Single-facet stabilization]
If the stabilization parameter is chosen as in \eqref{mixed-stab2}, then we can take 
 $F_K^*\in\EK$ in \eqref{broken-energy-estimate-m} to be the unique facet 
on which $\alpha_h$ is non-zero which implies that
\[
\sum_{K\in\Oh} h_K\bint{\alpha_h(u_h-\widehat u_h)}{u_h-\widehat u_h}{\dK\backslash F_K^*} = 0.
\]
Consequently, for the choice \eqref{mixed-stab2},
the estimator $\eta^2$ also 
gives a lower bound for the broken energy seminorm.
\end{remark}


\section{Numerical examples}
\label{sec:numerics}
In order to illustrate the results in Theorem \ref{thm:primal}--\ref{thm:mixed}, 
we consider Poisson problems in two and three dimensions approximated using the primal and mixed HDG schemes \eqref{primal-hdg} and \eqref{mixed-hdg}.
The implementation is performed using  the Python interface of the NGSolve software \cite{Schoberl97,Schoberl16}.
Conveniently, NGSolve provides 
a set of basis functions for the divergence-free bubble space $\Sigma_{k,sbb}$ \eqref{div-free} \cite{Zaglmayr06}
which makes the implementation of the equilibrated flux reconstructions \eqref{primal-flux-p} and \eqref{mixed-flux-p}  relatively straightforward.

\newcommand{\prpkI}{pr-P$_k$}
\newcommand{\prpkO}{pr-P$_k^{\mathrm{red}}$}
\newcommand{\mxpkI}{mx-P$_k$-U}
\newcommand{\mxpkO}{mx-P$_k$-S}

We choose the stabilization parameter for the primal HDG schemes \eqref{primal-hdg} to be 
\[
\alpha_h\restrict{F} =10 k^2 h_F^{-1}\quad \forall F\in \Eh
\]
{which, thanks to Lemma \ref{lemma:primal-stab}, ensures well-posedness on shape-regular meshes}.
We adopt a shorthand notation and denote the primal HDG scheme \eqref{primal-hdg} used  in 
conjunction with the approximation space  $\CV_{h,k,0}$ as \prpkI, 
{while if with the approximation space  $\CV_{h,k,1}$ used we use the notation: \prpkO.}  
Likewise, we denote the mixed HDG scheme \eqref{mixed-hdg} used in conjunction with uniform stabilization \eqref{mixed-stab1} as 
\mxpkI, while if the single-facet stabilization \eqref{mixed-stab2} is used, we write \mxpkO.
In all cases we use static condensation whereby the local, cell-wise, degrees of freedom 
are eliminated leaving only those degrees of freedom  which are located on the mesh skeleton. 

We take polynomial degree $k\in\{1,2,3,4\}$ for the first example, and 
$k\in\{1,2,3\}$ for the second example.

\subsection{Example 1: Two-dimensional L-shaped problem}
Here we consider the Laplace problem on a planar L-shaped domain $\Omega_{2D} = (-1, 1)\times (0,1)\cup (-1,0)\times (-1,0]$
with Dirichlet boundary conditions. 
The initial mesh is shown in Fig. \ref{fig1}(A).
The true solution is given by $u(r,\theta) = r^{2/3}\sin(2\theta/3)$ in polar coordinates.
\begin{figure}[!ht]
\subfloat[Example 1]{\includegraphics[width=0.45\textwidth]{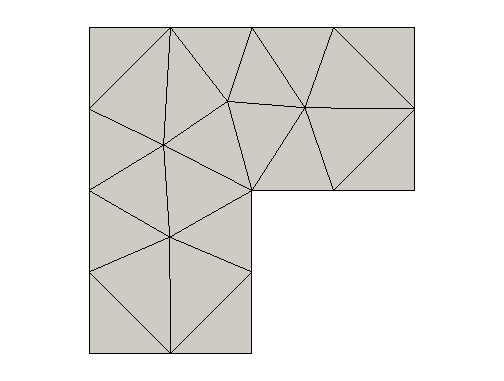}}
\subfloat[Example 2]{\includegraphics[width=0.45\textwidth]{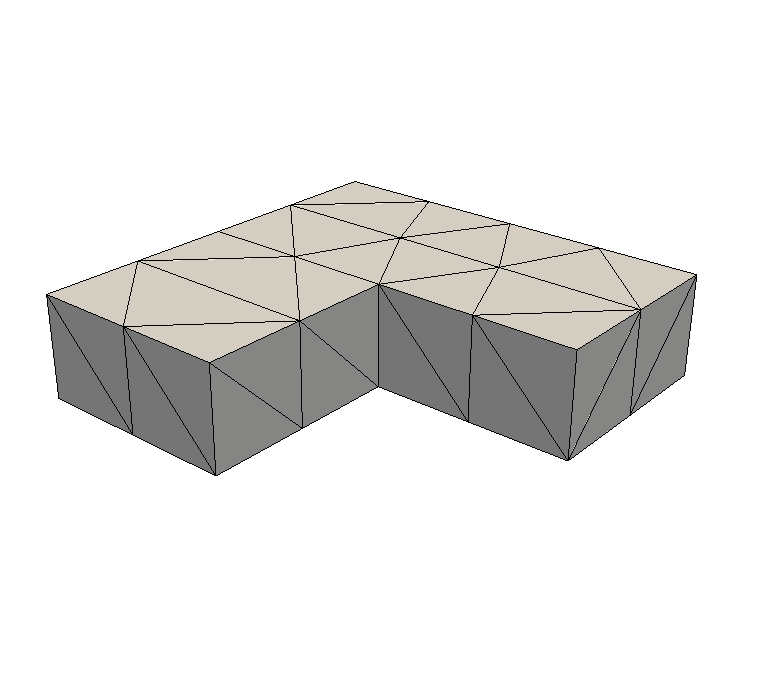}}
\caption{Initial meshes used in numerical examples.}
 \label{fig1}
\end{figure}

The sequence of meshes was constructed by selecting for refinement  the smallest number of elements 
whose combined contribution toward the estimator of the broken energy seminorm of the error exceeds half 
of the total estimated error. A sample of the meshes for the \prpkI\, scheme with $k = 1$ and $k = 4$
is shown in Fig. \ref{fig2}.
\begin{figure}[!ht]
\includegraphics[width=0.32\textwidth]{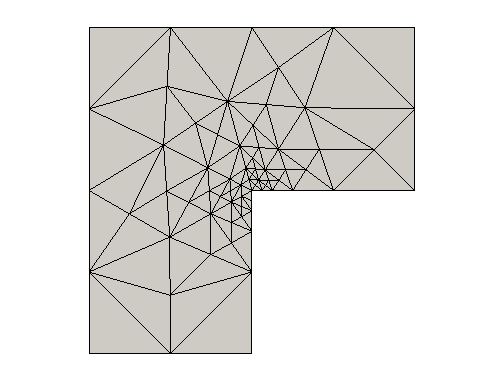}
\includegraphics[width=0.32\textwidth]{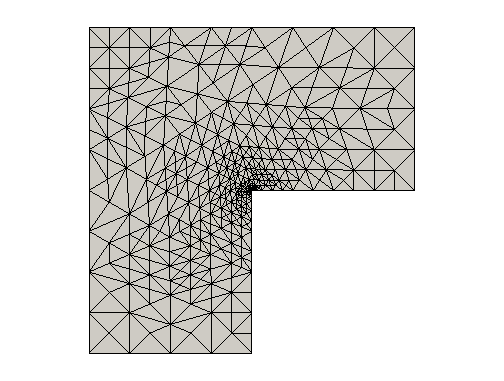}
\includegraphics[width=0.32\textwidth]{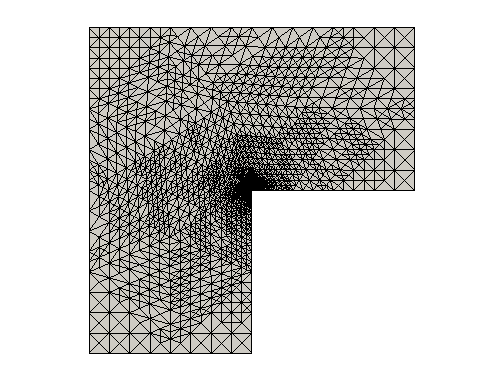}
\includegraphics[width=0.32\textwidth]{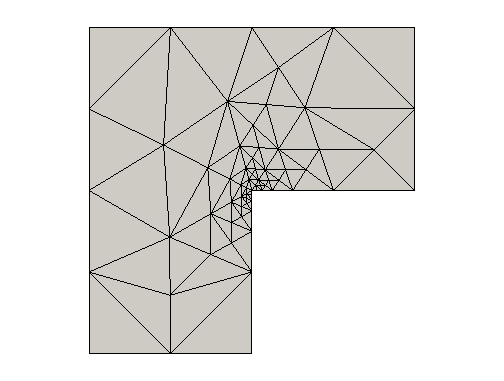}
\includegraphics[width=0.32\textwidth]{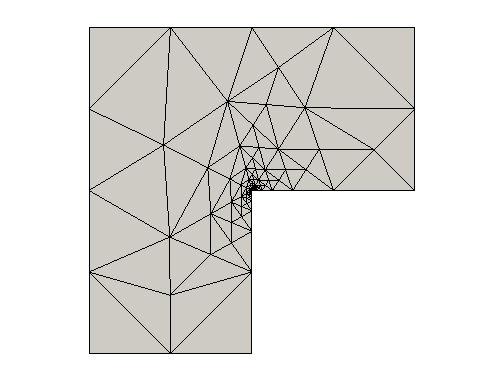}
\includegraphics[width=0.32\textwidth]{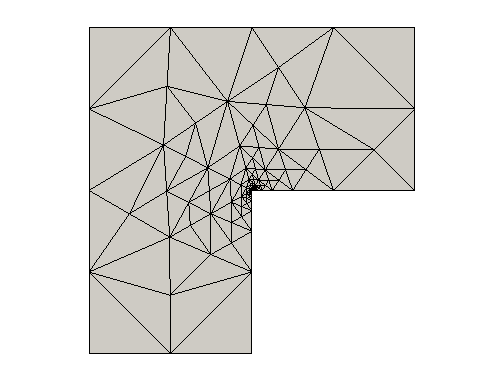}
\caption{The 4th, 9th and 13th meshes obtained performing adaptive refinement for Example 1.
Top: \prpkI\, scheme with $k=1$; Bottom: \prpkI\, scheme with $k=4$.}
 \label{fig2}
\end{figure}

In Fig. \ref{fig6}, 
we plot the error for the primal scheme in the broken energy seminorm \eqref{broken-h1} against the total number of degrees of freedom $\dim \CV_{h,k,\delta}$ and { the number of global  degrees of freedom, $\dim M_{h,k-\delta}^0$, remaining after local variables have been eliminated.}
The error in the broken energy seminorm \eqref{broken-h1-m} for the mixed HDG schemes is shown in Fig. \ref{fig7}.
\begin{figure}[!ht]
\includegraphics[width=\textwidth]{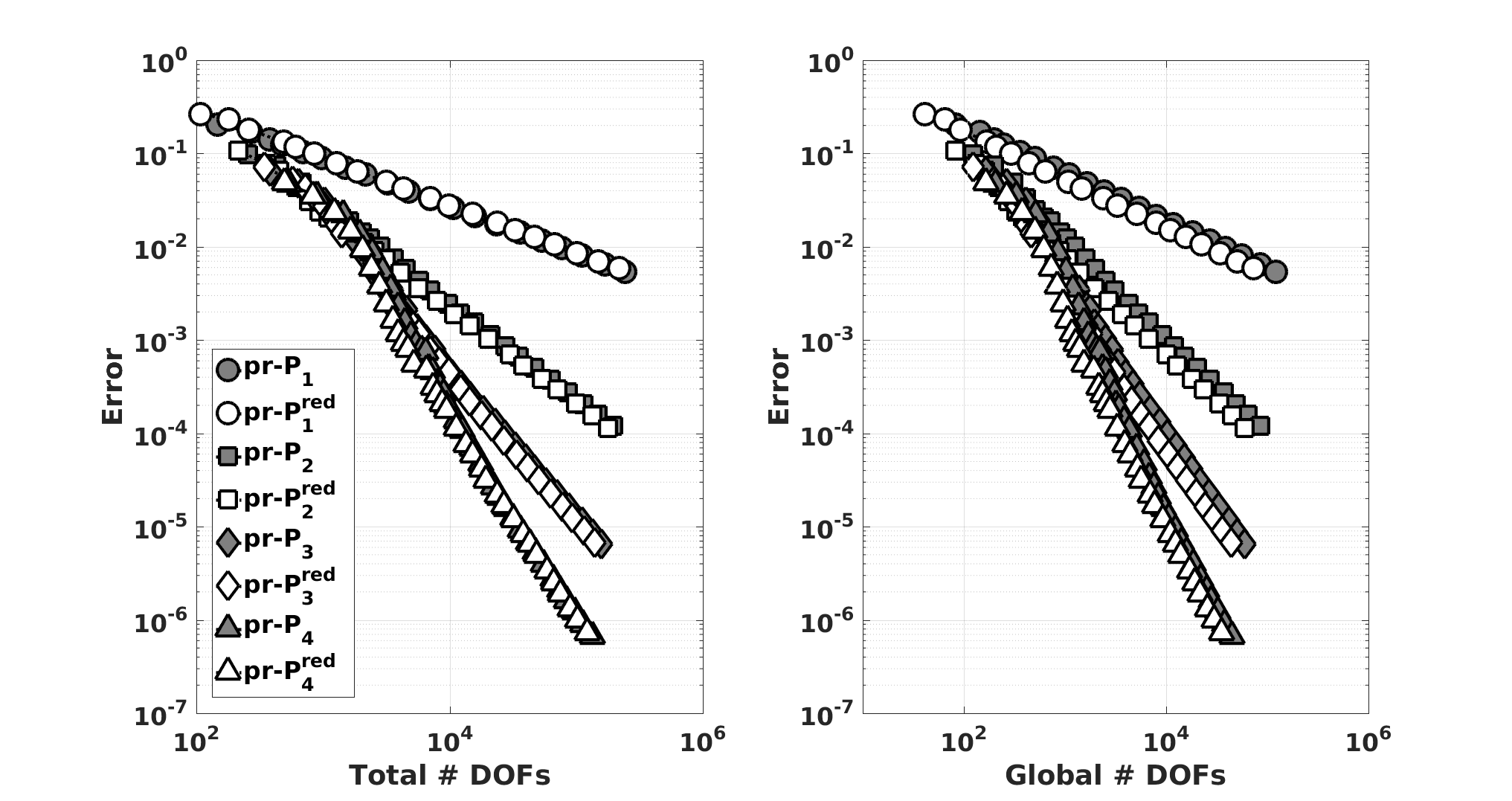}
\caption{Convergence history of the broken energy seminorm error for primal HDG schemes.
Left: error against total number of DOFs; Right: error against number of global skeleton DOFs.}
 \label{fig6}
\end{figure}
\begin{figure}[!ht]
\includegraphics[width=1.0\textwidth]{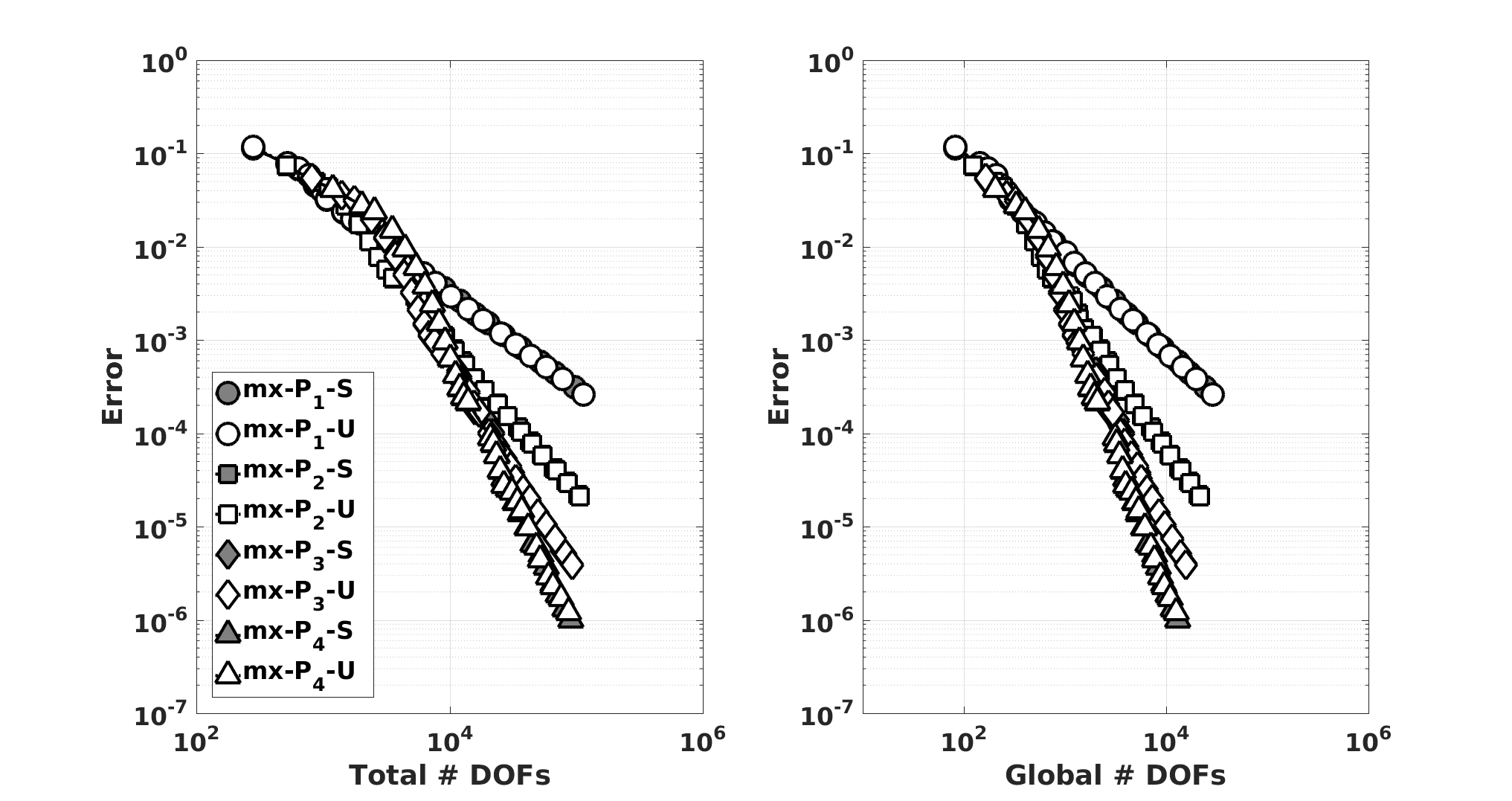}
\caption{Convergence history of the broken energy seminorm error for mixed HDG schemes.
Left: error against total number of DOFs; Right: error against number of global skeleton DOFs.}
 \label{fig7}
\end{figure}
In all cases, the effectivity indices are found to lie in the range $1.0$--$3.0$ as shown in Fig \ref{fig8}.
\begin{figure}[!ht]
\includegraphics[width=1\textwidth]{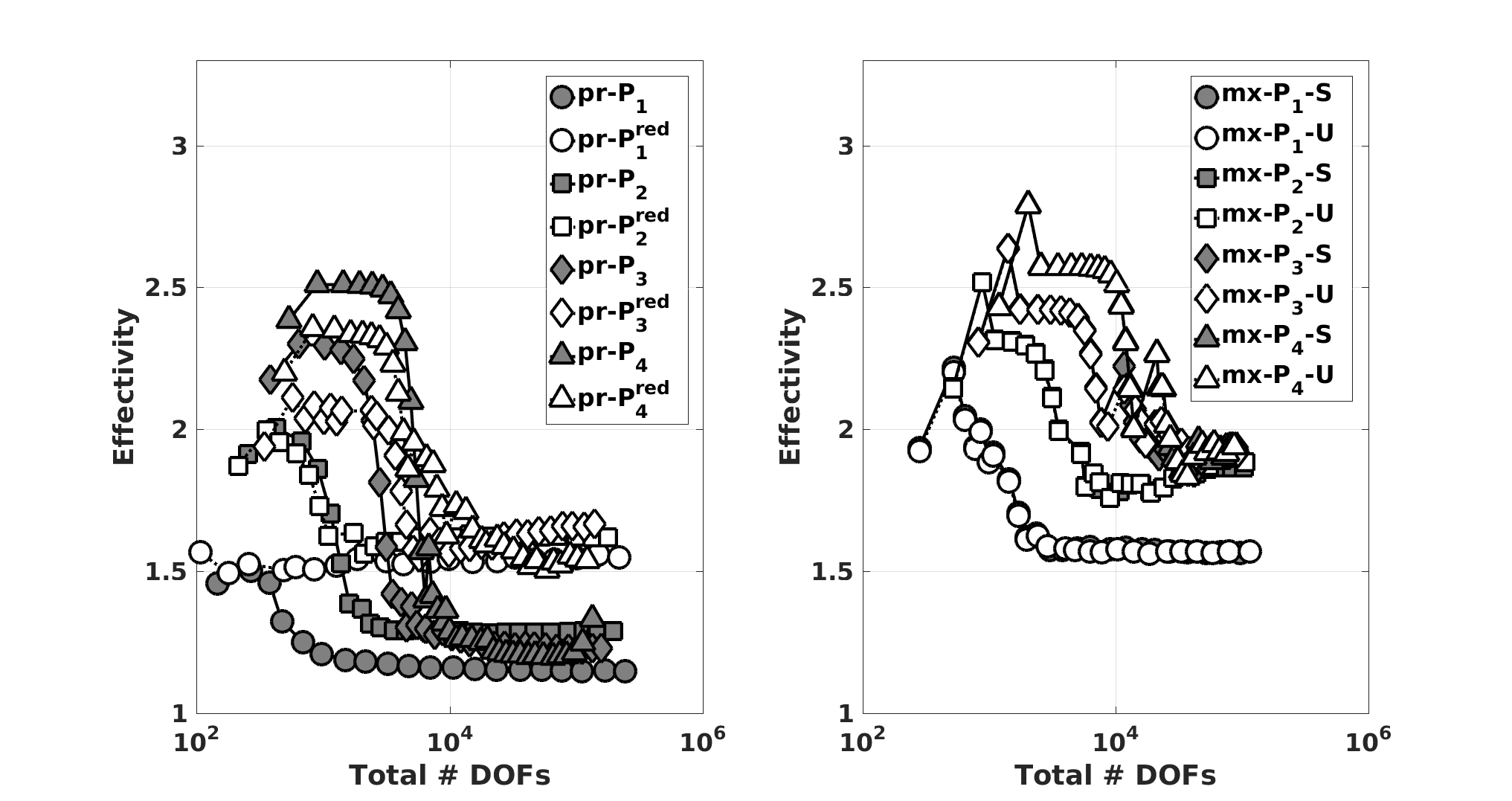}
\caption{History of effectivity indices $\eta/\text{error}$.
Left: primal HDG schemes; Right: mixed HDG schemes.}
 \label{fig8}
\end{figure}

\subsection{Example 2: Three-dimensional L-shaped problem}
Here we consider the Laplace problem on a three-dimensional L-shaped domain 
$\Omega_{3D} =\Omega_{2D}\times (0,0.5)$, where $\Omega_{2D}$ is the two-dimensional L-shaped domain in the previous example.
The initial mesh is shown in Fig. \ref{fig1}(B).
The true solution is independent of $z$, and reduces to the 
same two-dimensional solution as Example 1.
However, the fact that the true solution is independent of $z$ is not used in the finite element analysis and the meshes are unstructured through the thickness.

A sample of the meshes obtained for adaptive solution using the  \prpkI\, scheme with $k = 1$ and $k = 3$
is shown in Fig. \ref{fig3}.
The errors for the primal and mixed HDG schemes are plotted in  Fig. \ref{fig9} and Fig. \ref{fig10}, respectively.
As before, the effectivity indices are found to vary in the range $1.0$--$3.0$ as shown in Fig. \ref{fig11}.

\begin{figure}[!ht]
\includegraphics[width=0.45\textwidth]{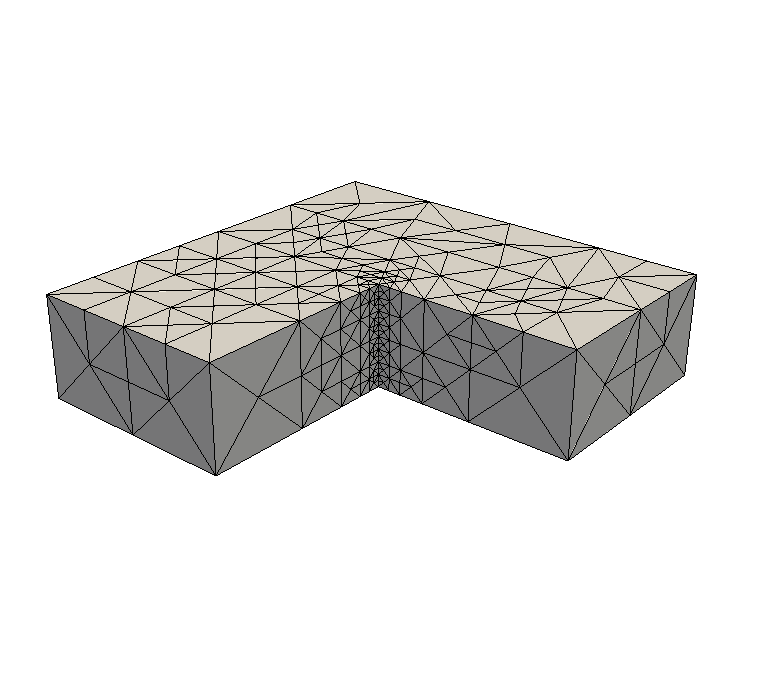}
\includegraphics[width=0.45\textwidth]{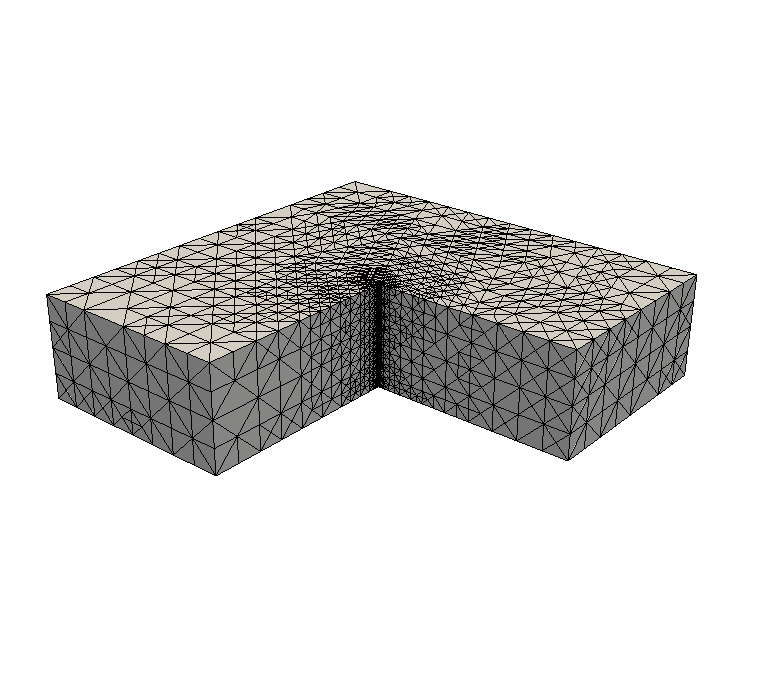}
\\
\vspace{-1.5cm}
\includegraphics[width=0.45\textwidth]{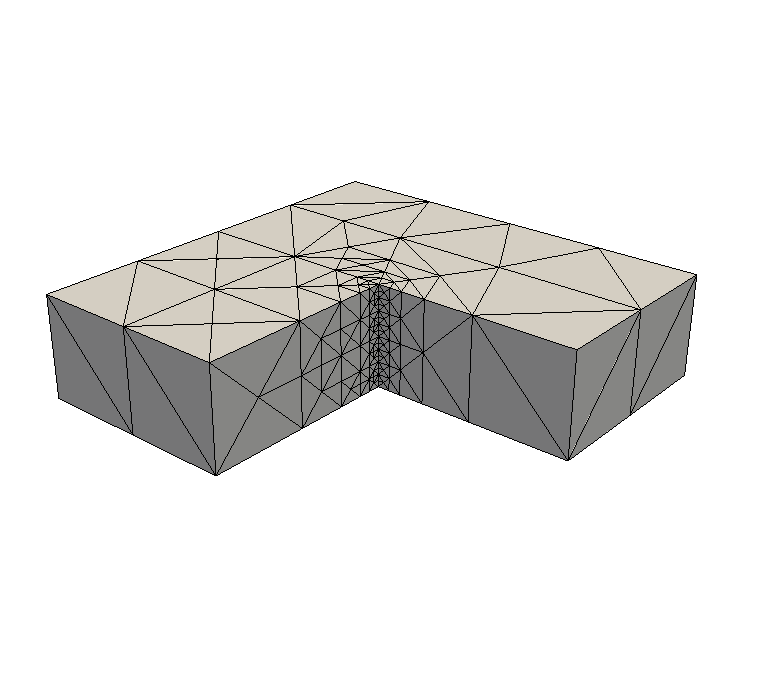}
\includegraphics[width=0.45\textwidth]{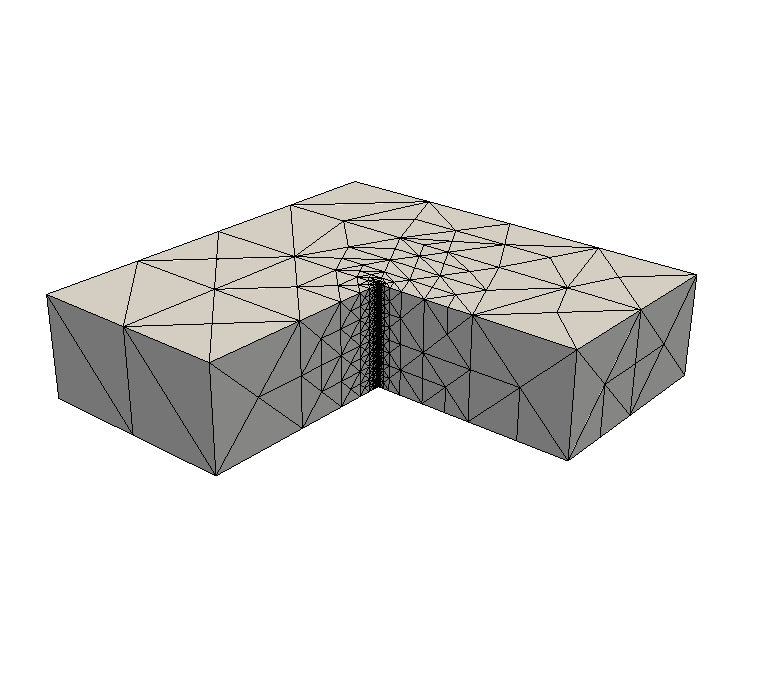}
\caption{The 4th and 7th adaptively refined meshes for Example 2.
Top: \prpkI\, scheme with $k=1$; Bottom:  \prpkI\, scheme with $k=3$.}
 \label{fig3}
\end{figure}
\begin{figure}[!ht]
\includegraphics[width=\textwidth]{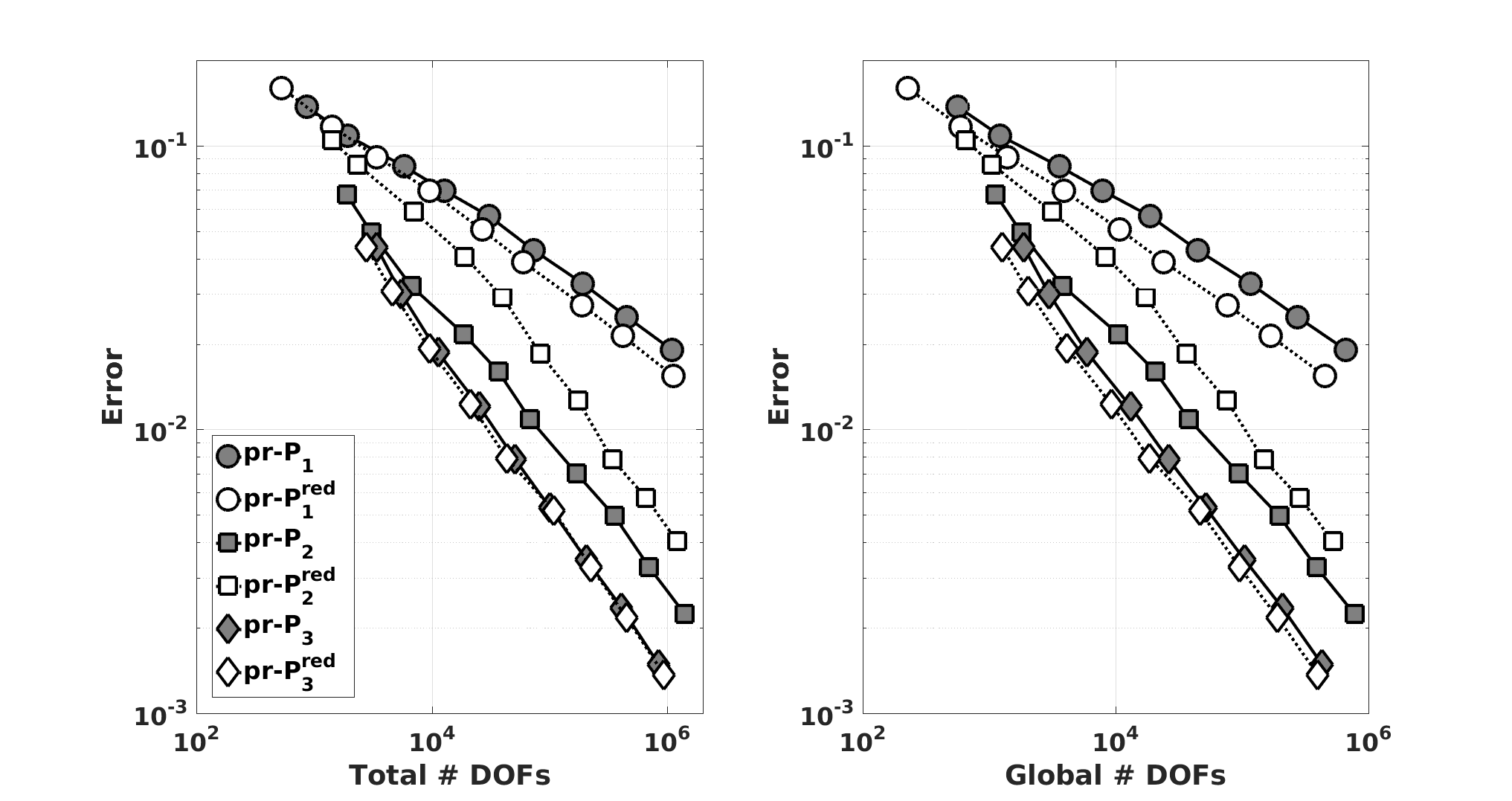}
\caption{Convergence history of the broken energy seminorm error for primal HDG schemes.
Left: error against total number of DOFs; Right: error against number of global DOFs.}
 \label{fig9}
\end{figure}
\begin{figure}[!ht]
\includegraphics[width=1.0\textwidth]{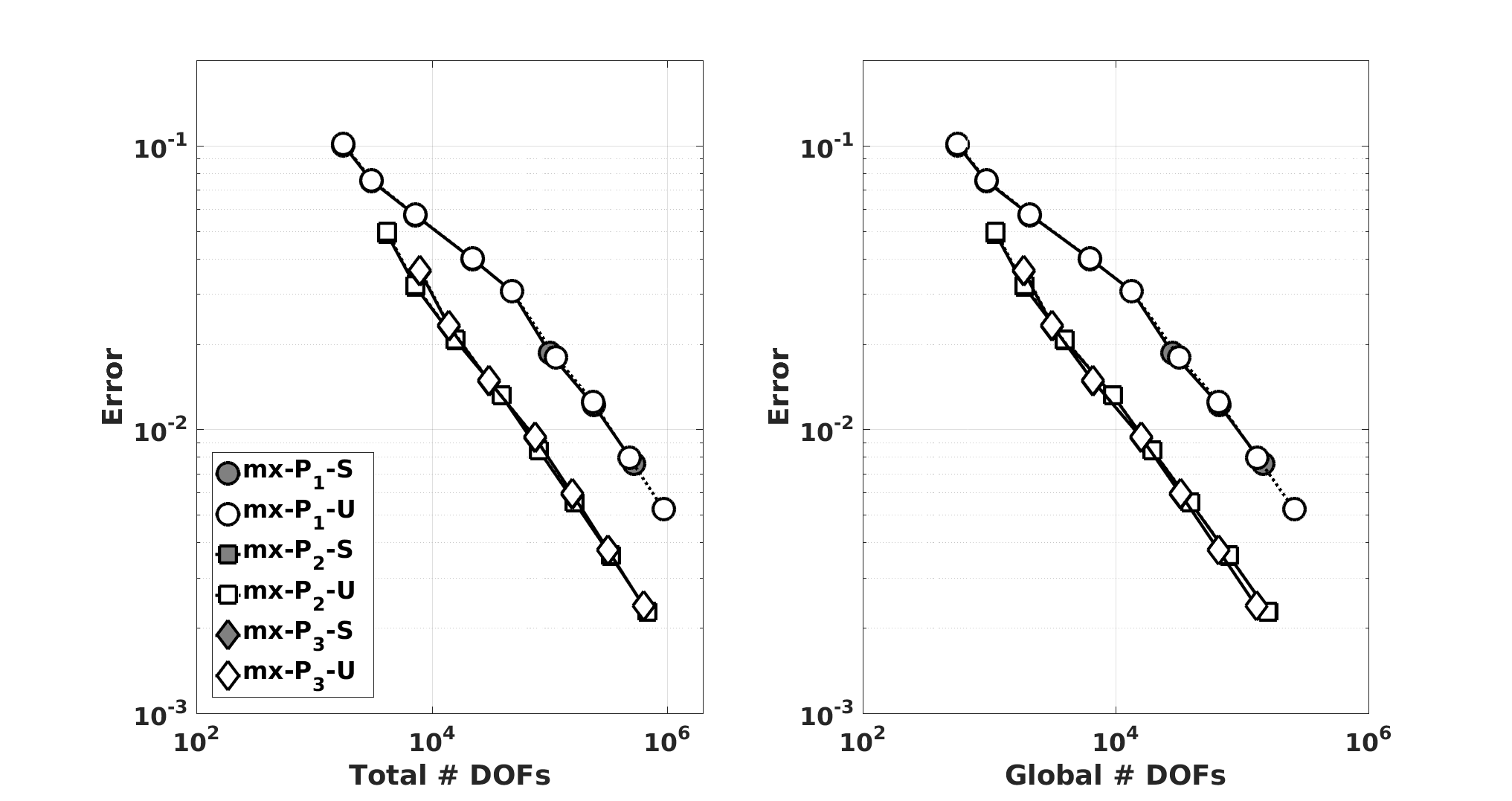}
\caption{Convergence history of the broken energy seminorm error for mixed HDG schemes.
Left: error against total number of DOFs; Right: error against  number of global DOFs.}
 \label{fig10}
\end{figure}
\begin{figure}[!ht]
\includegraphics[width=1\textwidth]{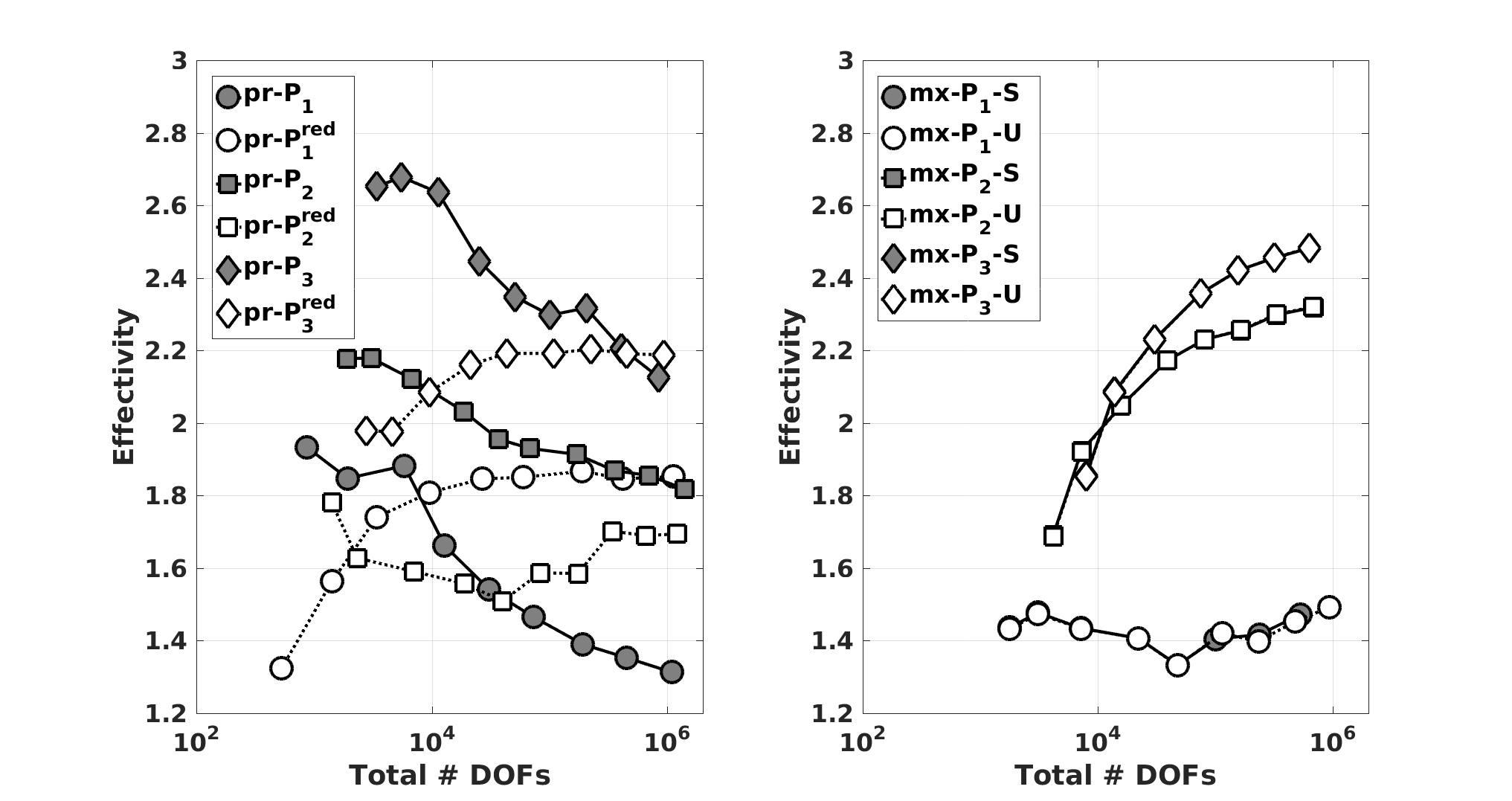}
\caption{History of effectivity indices $\eta/\text{error}$.
Left: primal HDG schemes; Right: mixed HDG schemes.}
 \label{fig11}
\end{figure}

\section{Proofs}
\label{sec:proofs}
{
We now turn to the proofs of the results.

As remarked earlier, the jump term 
in the HDG energy (semi)norm is directly computable, and, as such, 
we need only concern ourselves with obtaining estimates for
the broken energy seminorm of the error. 
To this end, recall the following Helmholtz decomposition \cite{GiraultRaviart86, DariDuranPadraVampa96}:
\begin{lemma}
 \label{lemma:helmholtz}
 Let $\Omega$ be a simply connected polygon/polyhedron.
Then, any $\b \tau\in L^2(\Omega)^d$, $d\in\{2,3\}$,
can be written in the form 
\begin{align}
\label{helmholtz}
 \b \tau = a\grads \phi + \curls \psi,
\end{align}
where $\phi\in H^1_0(\Omega)$ satisfies
\begin{align}
 \label{conforming}
 (a\grads \phi, \grads v) =  (\b \tau, \grads v)\quad\forall v\in H^1_0(\Omega).
\end{align}
and $\psi\in H^1(\Omega)^{2d-3}$ satisfies
\begin{align}
 \label{nonconforming}
 (a^{-1}\curls \psi, \curls \psi) =  
 (a^{-1}\b \tau, \curls \psi).
\end{align}
Moreover, the decomposition is orthogonal
\begin{align}
 \label{error-split}
 \|a^{-1/2}\b \tau\|^2
  = 
  \|a^{1/2}\grads \phi\|^2+
  \|a^{-1/2}\curls \psi\|^2.
\end{align}
\end{lemma}
We shall use the decomposition \eqref{helmholtz} in conjunction with $\b\tau$ defined elementwise by  $a\grads u -a\grads u_h$  for the primal HDG scheme \eqref{primal-hdg}, or by
$a\grads u-\b\sigma_h$ for the mixed HDG scheme \eqref{mixed-hdg}.
In both cases, Lemma \ref{lemma:helmholtz} gives an orthogonal decomposition of the broken energy seminorm error into the sum of a conforming part $\|a^{1/2}\grads \phi\|^2$ and a nonconforming part 
$\|a^{-1/2}\curls \psi\|^2$. It then suffices to obtain an a posteriori error bound for each part separately and sum to obtain an estimator for the total error.

\subsection{Proof of Theorem \ref{thm:primal}}
The proof of Theorem \ref{thm:primal} follows from 
\cite{AinsworthRankin11b} for the symmetric interior penalty discontinuous Galerkin methods almost verbatim.
Specifically, the upper bound \eqref{broken-energy-estimate} in Theorem \ref{thm:primal} follows from \eqref{ll1} and \eqref{ll2} below, and the lower bound \eqref{broken-energy-estimate2} follows from Lemma \ref{lemma:primal-norm-eq}, 
\eqref{ll3} and \eqref{ll4}.

The following estimates follows from  results in  
\cite[Lemma 6.2-6.5]{AinsworthRankin11b} using the proof in \cite[Section 6]{AinsworthRankin10a}.
 Let $\phi$ and $\psi$ be taken as  in  \eqref{helmholtz} 
 in the case where $\b\tau = a\grads u - a\grads u_h$, 
 and let 
 $\eta_{CF,K}$ and $\eta_{NC,K}$ be given by \eqref{indicator-p}.
Then,
\begin{subequations}
\label{ll}
\begin{align}
\label{ll1}
 \|a^{1/2}\grads \phi\|^2 \le &\;\sum_{K\in\Oh} \eta_{CF,K}^2,\\
\label{ll2}
 \|a^{-1/2}\curls \psi\|^2 \le &\;\sum_{K\in\Oh} \eta_{NC,K}^2.
\end{align}
Moreover,  there exists a positive constant $c$, 
depending only on the shape-regularity of the mesh and the polynomial degree $k$,  such that
\begin{align}
\label{ll3}
c\,\eta_{CF,K}^2\le &\;
a\,h_K^{-1}\|P_M(u_h-\widehat u_h)\|_{\dK}^2\nonumber\\
&+
\|a^{1/2}\grads \phi\|_{K}^2+osc_{k-1}^2(f, K),\\
\label{ll4}
c\,\eta_{NC,K}^2\le &\;
\sum_{F\in \widetilde{\mathcal{E}}_K}a\,{h_F^{-1}}{\left\|\jmp{u_h}\right\|_F^2},
\end{align} 
\end{subequations}
where 
$\widetilde{\mathcal{E}}_K
=\{F\in\partial \Oh: 
\overline F \cap \overline K \text{ is nonempty}\,
\}$.

\subsection{Proof of Theorem \ref{thm:mixed}}
Theorem \ref{thm:mixed} is a consequence of the  following three lemmas:

\begin{lemma}
 \label{lemma:mixed-key}
  Let $\phi$ and $\psi$ be given by in the  decomposition \eqref{helmholtz}, where $\b\tau$ is chosen to be $\b\tau = a\grads u -\b\sigma_h$, 
 and let 
 $\eta_{CF,K}$ and $\eta_{NC,K}$ be given by \eqref{indicator-m}.
Then, 
\begin{subequations}
\label{mm}
\begin{align}
\label{mm1}
 \|a^{1/2}\grads \phi\|^2 \le &\;\sum_{K\in\Oh} \eta_{CF,K}^2,\\
\label{mm2}
 \|a^{-1/2}\curls \psi\|^2 \le &\;\sum_{K\in\Oh} \eta_{NC,K}^2.
\end{align}
Moreover,  there exists a positive constant $c$, 
depending only on the shape-regularity of the mesh and the polynomial degree $k$, 
such that 
\begin{align}
\label{mm3}
c\,\eta_{CF,K}^2\le &\;
h_K \bintK{\alpha_h(u_h-\widehat u_h)}{u_h-\widehat u_h}\nonumber\\
&+\|a^{1/2}\grads \phi\|_{K}^2+osc_k^2(f, K),\\
\label{mm4}
c\,\eta_{NC,K}^2\le &\;
a|_K\left(\|\grads u_h^{*,\mathrm{dc}}-a^{-1}\b\sigma_h\|_K^2+
\sum_{F\in \widetilde{\mathcal{E}}_K}{\,h_F^{-1}}{\left\|\jmp{u_h^{*,\mathrm{dc}}}\right\|_F^2}\right),
\end{align} 
\end{subequations}
where 
$\widetilde{\mathcal{E}}_K
=\{F\in\partial \Oh: 
\overline F \cap \overline K \text{ is nonempty}\,
\}.$
\end{lemma}
\begin{proof}
Direct computation gives
\begin{alignat*}{3}
 \|a^{1/2}\grads \phi\|^2 = &\; 
 (a\grads u-\b\sigma_h, \grads \phi)\\
 = &\; 
 (f, \phi) -(\b\sigma_h, \grads \phi) &&\quad \text{(integration by parts)}\\
 = &\; 
 (f-\Pi_k f, \phi)
 +
 (\Pi_k f, \phi)\\
 &\;
 +\sum_{K\in\Oh}\bintK{\widehat{\b\sigma}_{h,mx}\cdot\b n}{\phi}
 -(\b\sigma_h, \grads \phi) &&\quad \text{(conservation \eqref{jmp-cont-m})}\\ 
 = &\; 
 (f-\Pi_k f, \phi)
 +(\b\sigma_h^*-\b\sigma_h, \grads \phi) &&\quad \text{(equilibration \eqref{eq-m})}\\  
 \le  &\; 
\eta_{CF,K}\|a^{1/2}\grads \phi\|,
\end{alignat*}
where the last inequality follows from the Cauchy-Schwarz and the Poincar\'e inequalities.
This completes the proof of \eqref{mm1}. Turning to \eqref{mm2}, since 
$ (\grads v, \curls \psi) = 0$ { for all $v\in H^1_0(\Omega)$}, 
we have 
\begin{alignat*}{3}
 \|a^{-1/2}\curls \psi\|^2 = &\; 
 (\grads u-a^{-1}\b\sigma_h, \curls \psi) = \; 
  (\grads u_h^*-a^{-1}\b\sigma_h, \curls \psi)\\
  \le &\;
\eta_{NC,K}\|a^{-1/2}\curls \psi\|.
\end{alignat*}

Let $\b\rho_h = \b\sigma_h^*-\b\sigma_h$, then for all $K\in\Oh$,
by \eqref{mixed-flux-p}, we have 
$\b\rho_h\in \pol_{k+1}(K)^d$ satisfies
\begin{align*}
(\divs \b\rho_h, v)_K = &\; 
-(f+\divs \b\sigma_h , v)_K
&&\quad \forall\, v \in \pol_{k}(K) \text{ and } (v,1)_K=0, \\
\bintK{\b\rho_h\cdot\b n}{\widehat v}
 = &\; 
 -\bintF{\alpha_h(u_h-\widehat u_h)}{\widehat v}
 &&\quad \forall\, \widehat v \in \pol_{k+1}(F),\quad \forall F\in \EK,\\
(\b\rho_h, \b\tau)_K = &\; 
0 &&\quad \forall\, \b\tau \in 
\Sigma_{k+1, sbb}.
\end{align*}
Hence, there exists a constant $c$, depending only on the polynomial degree $k$ and shape-regularity of the element $K$, such that
\begin{align}
  c\|\rho_h\|_K^2\le&\; h_K^2\|f+\divs \b\sigma_h\|_K^2 + h_K \|\alpha_h(u_h-\widehat u_h)\|_{\dK}^2,
\end{align} 
while, using a standard bubble function technique \cite{AinsworthOden00,Verfurth96}, 
we have 
\begin{align}
\label{bubble}
 c\,h_K^2\|f+\divs \b\sigma_h\|_K^2 \le 
\|a\,\grads \phi\|_{K}^2+ a|_K\,osc_k^2(f,K).
\end{align}
The choice of stabilization parameter \eqref{mixed-stab} means that 
\[
\|\alpha_h(u_h-\widehat u_h)\|_{\dK}^2
= \,a \restrict K\, \|\alpha_h^{1/2}(u_h-\widehat u_h)\|_{\dK}^2,
\]
and the proof of  \eqref{mm3} follows from these estimates.

Finally, we have 
\begin{align*}
\|\grads u_h^* - a^{-1}\b\sigma_h\|_K^2 \le &\;
2(\|\grads u_h^* - \grads u_h^{*,\mathrm{dc}}\|_K^2
+
\|\grads u_h^{*,\mathrm{dc}} - a^{-1}\b\sigma_h\|_K^2),\\
\|\grads u_h^* - \grads u_h^{*,\mathrm{dc}}\|_K^2\le &\;c \sum_{F\in \widetilde{\mathcal{E}}_K}{h_F^{-1}}{\left\|\jmp{u_h^{*,\mathrm{dc}}}\right\|_F^2}.
\end{align*}
Combining the above estimates completes the proof of  \eqref{mm4}.
\end{proof}

\begin{lemma}
 \label{lemma:mixed-jmp}
There exists a positive constant $c$, 
depending only on the shape-regularity of the mesh and the polynomial degree $k$, 
such that for any facet $F\in\EK$,  
 \begin{align}
ch_K\|\alpha_h^{1/2}(u_h-\widehat u_h)\|_F^2
 \le &\;
h_K \|\alpha_h^{1/2}(u_h-\widehat u_h)\|_{\dK\backslash F}^2\nonumber\\
 &\hspace{-0.cm}
 +\|a^{1/2}\,\grads \phi\|_{K}^2+osc_k^2(f,K).
 \end{align} 
\end{lemma}
\begin{proof}
The mixed HDG scheme \eqref{mixed-hdg} satisfies, for every $K\in \Oh$, 
\begin{align}
 \label{single-eq}
\bintK{\alpha_h(u_h-\widehat u_h)}{v} = (f+\divs \b\sigma_h, v)_K\quad 
\text{ for all }v\in \pol_k(K).
\end{align}
Let the function $z\in \pol_k(K)$ satisfy 
\begin{alignat*}{3}
 (z,w)_K  = &\; 0 \quad \text{ for all }w\in\pol_{k-1}(K),\\
 \bint{z}{\widehat w}{F^*}  = &\; \bint{\alpha_h(u_h-\widehat u_h)}{\widehat w}{F^*} \quad \text{ for all }w\in\pol_{k}(F^*),
\end{alignat*}
where $F^*$ is a fixed facet of $K$.
We have, by a standard scaling argument, 
\[
\|z\|_K\le c\, h_{K}^{1/2} \|\alpha_h(u_h-\widehat u_h)\|_{F^*}.
\]
Taking $v=z$ in \eqref{single-eq} and rearranging terms,  
we obtain 
\begin{align*}
\|\alpha_h(u_h-\widehat u_h)\|^2_{F^*} = &\;
(f+\divs \b\sigma_h, z)_K
+\bint{\alpha_h(u_h-\widehat u_h)}{z}{\dK\backslash F^*} \\
\le &\;
c\left(\|f+\divs \b\sigma_h\|_K
+h_K^{-1/2}\|\alpha_h(u_h-\widehat u_h)\|_{\dK\backslash F^*}\right)\|z\|_K\\
 &\hspace{-2cm}
\le\;c\left(h_K^{1/2}\|f+\divs \b\sigma_h\|_K
+\|\alpha_h(u_h-\widehat u_h)\|_{\dK\backslash F^*}\right)
\|\alpha_h(u_h-\widehat u_h)\|_{F^*}
\end{align*}
The proof is completed by invoking estimate \eqref{bubble} for the cell-wise residual term $f+\divs \b\sigma_h$.
\end{proof}

\begin{lemma}
 \label{lemma:mixed-nonconforming}
There exists a positive constant $c$,
depending only on the shape-regularity of the mesh and the polynomial degree $k$, 
such that 
 \begin{align}
 \label{non-mix1}
 c \|\grads u_h^{*,\mathrm{dc}}-a^{-1}\b\sigma_h\|_K^2\le &\;
 \|\grads u-a^{-1}\b\sigma_h\|_K^2,\\  
 \label{non-mix2}
c{h_F^{-1}}{\left\|\jmp{u_h^{*,\mathrm{dc}}}\right\|_F^2}
\le &
\;
\sum_{K'\in \widetilde F}\|\grads u-a^{-1}\b\sigma_h\|_{K'}^2.
 \end{align} 
\end{lemma}


\begin{proof}
We first prove the estimate \eqref{non-mix1}.
We denote $\b\rho_h  = \grads u_h^{*,\mathrm{dc}}-a^{-1}\b\sigma_h$ and let
 the function $\b \rho_h^* \in \pol_{k+1}(K)^d\oplus \b x\widetilde{\pol}_{k+1}(K)$
be defined as follows:
\begin{subequations}
\label{rt}
\begin{align}
\label{rt1}
\bintF{\b\rho_h^*\cdot\b n}{\widehat v}
 = &\; 
0
 &&\quad \forall\, \widehat v \in \pol_{k+1}(F),\quad \forall F\in \EK,\\
\label{rt2}
(\b\rho_h^*, \b v)_K = &\; 
(\b \rho_h , \b v)_K
&&\quad \forall\, \b v \in 
\pol_{k}(K)^d,
\end{align}
\end{subequations}
then,  we have 
\begin{align}
\label{rt-p}
c \|\b\rho_h^*\|_K\le \|\b\rho_h\|_K.
\end{align}
Moreover, by equations \eqref{mixed-potential-p-1} and \eqref{rt}, we have
\[
(\divs \b\rho_h^*, v)_K = 
-(\b\rho_h^*,\grads v)_K
=
-(\b\rho_h,\grads v)_K = 0\quad \text{ for all }v\in \pol_{k+1}(K).
\]
This implies that $\divs \b\rho_h^* = 0$
because $\divs \b\rho_h^*\in \pol_{k+1}(K)$.
Since $\b\rho_h^*$ has vanishing normal trace on $\dK$ by 
equations \eqref{rt1}, we obtain
\[
 (\b\rho_h^*, \grads v)_K = 0\quad \text{ for all }v\in H^1(K).
\]
Hence, 
\begin{align*}
 \|\b\rho_h\|_K^2 = 
  (\b\rho_h^*, \b\rho_h)_K 
  = 
  (\b\rho_h^*, \grads u_h^{*,\mathrm{dc}} - a^{-1}\b\sigma_h)_K= 
    (\b\rho_h^*, \grads u - a^{-1}\b\sigma_h)_K
\end{align*}
The estimate \eqref{non-mix1} now follows from  the 
Cauchy-Schwarz inequality.

Let $P_{M_0}$ be the $L^2$-projection onto the space $\pol_0(F)$.
Applying the results in \cite[Lemma 3.4-3.5]{CockburnZhang14}, we obtain 
\begin{align*}
h_F^{-1}\left\|\jmp{u_h^{*,\mathrm{dc}}}\right\|_F^2 = &\;
h_F^{-1}\left\|P_{M_0}\jmp{u_h^{*,\mathrm{dc}}}\right\|_F^2
+h_F^{-1}\left\|(\mathrm{Id}-P_{M_0})\jmp{u_h^{*,\mathrm{dc}}}\right\|_F^2\\
\le &\; 
c\,\sum_{K\in\widetilde F}
(\|a^{-1}\b\sigma_h-\grads u_h^{*,\mathrm{dc}}\|_{K}^2
+\|\grads u-\grads u_h^{*,\mathrm{dc}}\|_{K}^2).
\end{align*}
The estimate \eqref{non-mix2} now immediately follows from 
the triangle inequality and \eqref{non-mix1}.

\end{proof}


\subsection{Proof of Lemma \ref{lemma:primal-stab}}
Since $\CV_{h,k, \delta}$ is a finite-dimensional space, 
 it suffices to show that $(u_h, \widehat u_h) = (0,0)$ is the only solution to the homogeneous problem.
Let
\[
 C_{k,K}: =\sum_{F'\in\EK}\frac{|F'|^2}{|K|},\text{ and }   D_{k,K}: = \frac{k(k+d-1)}{d}C_{k,K}, 
\]
then for $(v_h,\widehat v_h)\in \CV_{h,k, \delta}$, we have 
 \begin{align*}
  \Bpr\big(
  (v_h,\widehat v_h),(v_h,\widehat v_h)
  \big) = &\;
 \sum_{K\in\Oh}\Big\{
(a\,\grads v_h, \grads v_h)_K - 2\bintK{a\,\grads v_h\cdot \b n}{v_h-\widehat v_h\,}\\
&\;
+ \bintK{\frac{a\,\gamma\, C_{k,K}}{|F|}(P_M v_h-\widehat v_h)}{P_Mv_h-\widehat v_h\,} 
 \Big\},
 \end{align*}
 Since $a\grads v_h\cdot \b n\restrict F\in \pol_{k-1}(F)$, there holds 
 \[
  \bintK{a\,\grads v_h\cdot \b n}{v_h-\widehat v_h\,}
  =\bintK{a\,\grads v_h\cdot \b n}{P_Mv_h-\widehat v_h\,}.
 \] 
 By the Cauchy-Schwarz and Young's inequalities,  for any $\epsilon >0$,
 \begin{align*}
  2\bintF{a\,\grads v_h\cdot \b n}{v_h-\widehat v_h\,}
  \le &\;
 \frac{|F|}{a\,\epsilon} \|a\,\grads v_h\cdot \b n\|_{F}^2
+\frac{a\,\epsilon}{|F|} \|P_Mv_h-\widehat v_h\|_{F}^2\\
&\!\!\!\!\!\!\!\!\!\!\!\!\!\!\!\!\!\!\!\!\!\!\!\!\!\le \;
 \frac{|F|}{a\,\epsilon} \frac{k(k+d-1)}{d}\frac{|F|}{|K|}
\|a\,\grads v_h\|_{K}^2    
+\frac{a\,\epsilon}{|F|} \|P_Mv_h-\widehat v_h\|_{F}^2
 \end{align*}
where, 
the final inequality holds thanks to 
{ the inverse-trace inequality \cite{WarburtonHesthaven03}
and $a\grads u\restrict K \in \pol_{k-1}(K)^d$}.
Summing the above inequality over $F\in \EK$ gives \begin{align*}
  2\bintK{a\,\grads v_h\cdot \b n}{v_h-\widehat v_h\,}
  \le &\;
 \frac{D_{k,K}}{\epsilon} \|a^{1/2}\,\grads v_h\|_{K}^2
+\sum_{F\in\EK}\frac{a\,\epsilon}{|F|} \|P_Mv_h-\widehat v_h\|_{F}^2.
 \end{align*}
Hence, we have 
\begin{align*}
   \Bpr\big(
  (v_h,\widehat v_h),(v_h,\widehat v_h)
  \big)
  \ge &\;
  \sum_{K\in\Oh}
\big(1-\frac{D_{k,K}}{\epsilon}
\|a^{1/2}\,\grads v_h\|_{K}^2\\
&\;\;
+
\sum_{F\in\EK}\frac{a\,(\gamma \, C_{k,K} -\epsilon)}{|F|} \|P_Mv_h-\widehat v_h\|_{F}^2
\Big).
\end{align*}

Finally, if $\gamma$ satisfies \eqref{stab-cst} then there exists
an  $\epsilon$ such that 
\[
\gamma\, C_{k,K} >\epsilon >D_{k,K}.
\]
Consequently, when the right hand side of \eqref{primal-hdg} vanishes, 
$\|\grads v_h\|_K = 0$ for all $K\in \Oh$ and 
$\|P_M v_h-\widehat v_h\|_F = 0$ for all $F\in \Eh$.
Hence, the only solution to the homogeneous problem is $(v_h, \widehat v_h) =(0,0)$, and therefore there exists a unique solution $(u_h,\widehat u_h)\in \CV_{h,k,\delta}$ to \eqref{primal-hdg}.

\subsection{Proof of Lemma \ref{lemma:primal-norm-eq}}
The proof follows that of  \cite[Theorem 3]{Ainsworth07a}. 

By the conservation property \eqref{jmp-cont-p}, we have 
\begin{align}
\label{explicit}
 \widehat u_h\restrict F = \left\{
 \begin{tabular}{l l}
  $0$ & if $F\in \Eh^\partial$,\vspace{0.2cm}\\
  $\frac{\avg{\alpha_h\, P_Mu_h}}{\avg{\alpha_h}}-\frac{1}{2\avg{\alpha_h}}{\jmp{a\grads u_h}}$ & if $F\in \Eh^o$.
 \end{tabular}
\right.
\end{align}
Hence, 
 \begin{align*}
P_M u_h- \widehat u_h\restrict F = \left\{
 \begin{tabular}{l l}
  $P_M u_h$ & if $F\in \Eh^\partial$,\vspace{0.2cm}\\
  $\frac{\jmp{P_Mu_h}\cdot\b n}{2\alpha_h\avg{1/\alpha_h}}+\frac{1}{2\avg{\alpha_h}}{\jmp{a\grads u_h}}$ & if $F\in \Eh^o$.
 \end{tabular}
\right.
\end{align*}
Inserting  the above expression into the jump term in the HDG energy norm and regrouping gives
\begin{align*}
 \sum_{K\in\Oh} 
  \bintK{\alpha_h P_M(u_h-\widehat u_h)}{P_M(u_h-\widehat u_h)}
  & \\
  &\hspace{-5cm} = 
  \sum_{F\in\Eh^o}\left(
  \bintF{\frac{1}{2\avg{1/\alpha_h}} \jmp{P_M u_h}}{\jmp{P_M u_h}}
  +
\bintF{\frac{1}{2\avg{\alpha_h}}\jmp{a\grads u_h}}{\jmp{a\grads u_h}}\right)\\
  &\hspace{-4cm} + 
  \sum_{F\in\Eh^\partial}
  \bintF{\alpha_h {\jmp{P_M u_h}}}{\jmp{P_M u_h}}
\end{align*}
Here, the gradient jump term can be controlled by the standard bubble function technique \cite{Verfurth96,AinsworthOden00}
\begin{align*}
c  \sum_{F\in\Eh^o}
 h_F\| \jmp{a\grads u_h}\|_F^2\le &\;
  \vertiii{(e_u,\widehat e_u)}_{pr}^2+ \sum_{K\in\Oh} osc_{k-1}^2(f,K).
\end{align*}

On the other hand, there holds
\[
\|\jmp{P_M u_h}\|_F^2
 = 
 |F|(\overline{\jmp{P_M u_h}})^2
  + 
\|\jmp{P_M u_h}-\overline{\jmp{P_M u_h}}\|_F^2,
\]
where $\overline{\jmp{P_M u_h}}$ denote the average value of $\jmp{P_M u_h}$ on the facet $F$.
Thanks to the trace and Poincar\'e inequalites, we have 
\[
c h_F^{-1}\|\jmp{P_M u_h}-\overline{\jmp{P_M u_h}}\|_F^2
 \le \sum_{K'\in\widetilde F}\|\grads u-\grads u_h\|_{K'}^2.
\]

Hence, to show norm equivalence, it remains to show that the term
\[
\sum_{F\in\Eh} \frac{|F|}{h_F}(\overline{\jmp{P_M u_h}})^2
\]
can be controlled by the discrete energy seminorm plus the data oscillation. 
Replacing $\widehat u_h$ in the primal HDG scheme 
\eqref{primal-hdg} with the expression \eqref{explicit}, we obtain 
\begin{align}
 \label{single}
0= &\;(f-\divs a\grads u_h, v_h)_K\nonumber\\
 &\;-\sum_{F\in\EKI}\bintF{\frac{\jmp{P_Mu_h}\cdot\b n}{2\alpha_h\avg{1/\alpha_h}}+\frac{1}{2\avg{\alpha_h}}\jmp{a\grads u_h}}{a\grads v_h\cdot\b n-\alpha_h\,P_M v_h}\nonumber\\
 &\;-\sum_{F\in\EKB}\bintF{P_Mu_h}{a\grads v_h\cdot\b n-\alpha_h\,P_M v_h}
\end{align}
for all $v_h\in V_{h,k}$ and $K\in\Oh$.
{
Moreover, there holds \cite{Shewchuk03}
\[
 (d+1)^2\rho(\b S_K) \le \sum_{F'\in\EK}\frac{|F'|^2}{|K|},
\]
where $\b S_K$ is the element stiffness matrix, i.e. $\b S_{ij}=(\grads \lambda_i,\grads\lambda_j)_K$ with $\{\lambda_\ell\}_{\ell=1}^{d+1}$ being the barycentric coordinates for the element $K$, and 
$\rho(\b S_K)$ is its spectral radius. 
Hence the stabilization parameter $\alpha_h$ in \eqref{stab-pr}, with $\gamma$ satisfying \eqref{stab-cst}, satisfies
\[
\alpha_h\restrict{F} > \frac{a|_K}{|F|}(d+1)^2\rho(\b S_K),\quad 
\forall F\in\EK.
\]
The proof is then concluded following \cite[Theorem 3]{Ainsworth07a} by taking special linear test functions in the equation 
\eqref{single} and using the above estimate for the stabilization parameter.}

\bibliographystyle{siam}

\end{document}